\documentclass[14pt,a4paper,reqno]{amsart}
\usepackage{color}
\usepackage{amsfonts,amsmath,amssymb,amsxtra,url,float}
\allowdisplaybreaks[4]
\usepackage[colorlinks,linkcolor=RoyalBlue,anchorcolor=Periwinkle,citecolor=Orange,urlcolor=Green]{hyperref}
\usepackage[usenames,dvipsnames]{xcolor}
\usepackage{geometry} 
\usepackage{rotating}
\usepackage{lscape}
\usepackage{enumitem}
\DeclareSymbolFont{largesymbol}{OMX}{yhex}{m}{n}
\DeclareMathAccent{\Widehat}{\mathord}{largesymbol}{"62}
\usepackage{multirow}
\usepackage{graphicx}
\usepackage{subfigure}
\usepackage{tikz}
\usetikzlibrary{decorations.pathreplacing,decorations.markings}
\tikzset{
on each segment/.style={
decorate,
decoration={
show path construction,
moveto code={},
lineto code={
\path [#1]
(\tikzinputsegmentfirst) -- (\tikzinputsegmentlast);
},
curveto code={
\path [#1] (\tikzinputsegmentfirst)
.. controls
(\tikzinputsegmentsupporta) and (\tikzinputsegmentsupportb)
..
(\tikzinputsegmentlast);
},
closepath code={
\path [#1]
(\tikzinputsegmentfirst) -- (\tikzinputsegmentlast);
},
},
},
mid arrow/.style={postaction={decorate,decoration={
markings,
mark=at position 0.6 with {\arrow[#1]{stealth}} %
}}},
}
\usetikzlibrary{arrows}
\usetikzlibrary{trees}

\usetikzlibrary{matrix}
\usetikzlibrary{patterns}
\usetikzlibrary{shadings}
\usepackage[all]{xy}
\usepackage{mathdots}
\usepackage{dsfont}
\usepackage{cite}
\usepackage{mathrsfs}
\usepackage{multicol}
\numberwithin{figure}{section}
\usepackage{marginnote}
\usepackage{graphicx}
\usepackage{multicol}

\usepackage{fancyhdr}
\newcommand{\checks}[1]{{\color{black}{#1}}} %
\newcommand{\checkLiu}[1]{{\color{black}{#1}}}
\newtheorem{theorem}{Theorem}[section]
\newtheorem{lemma}[theorem]{Lemma}
\newtheorem{corollary}[theorem]{Corollary}
\newtheorem{main theorem}[theorem]{Main Theorem}
\newtheorem{proposition}[theorem]{Proposition}
\newtheorem{definition}[theorem]{Definition}

\newtheorem{remark}[theorem]{Remark}
\newtheorem{example}[theorem]{Example}

\usetikzlibrary{arrows}
\numberwithin{equation}{section}
\def\<{\langle}
\def\>{\rangle}
\def\NN{\mathbb{N}}

\def\RR{\mathbb{R}}
\def\II{\mathbb{I}}
\newcommand{\Pic}{F{\tiny{IGURE}}\ }

\newcommand{\kk}{\mathds{k}}
\newcommand{\Q}{\mathcal{Q}}
\newcommand{\I}{\mathcal{I}}

\newcommand{\Hom}{\mathrm{Hom}}
\newcommand{\End}{\mathrm{End}}
\newcommand{\Ext}{\mathrm{Ext}}
\newcommand{\op}{\mathrm{op}}

\newcommand{\sectcolor}{}
\def\s{\mathfrak{s}}
\def\t{\mathfrak{t}}

\def\itLamb{\mathit{\Lambda}}
\def\im{\mathrm{Im}}

\def\bfS{\mathbf{S}}

\def\id{\mathbf{1}}
\def\scrN{\mathscr{N}\!\!or}
\def\scrA{\mathscr{A}}

\def\frakU{\mathfrak{U}}
\def\frakS{\mathfrak{S}}
\def\spI{\mathfrak{I}}
\def\spC{\pmb{\Game}}

\def\dd{\mathrm{d}}
\def\rad{\mathrm{rad}}
\def\iso{\omega}
\def\ImspC{{{}^{\Im}\mkern-1mu\pmb{\Game}}}
\def\fct{\!{\lower 0.8ex\hbox{\tikz\draw (0pt, 0pt) node{\scriptsize$\mathfrak{F}$};}}\!}

\def\isoinv{\varpi}

\def\heart{{\color{red}\pmb{\heartsuit}}}
\def\diamo{{\color{red}\pmb{\diamondsuit}}}

\newcommand{\circled}{\ \lower-0.2ex\hbox{\tikz\draw (0pt, 0pt) circle (.1em);} \ }

\newcommand{\Norm}[2]{\Vert#1\Vert_{#2}}

\newcommand{\w}[1]{\widehat{#1}}

\newcommand{\iposet}{i.poset\ }

\newcommand{\defines}{\it}

\begin{document}

\def\headertitle{Normed modules, integral sequences, and integrals with variable upper limits}

\title[\headertitle]{Normed modules, integral sequences, and integrals with variable upper limits}

\def\fstpage{1}
\def\page{$\begin{matrix} {\color{white}0} \\ \thepage \end{matrix}$}
\pagestyle{fancy}
\fancyhead[LO]{ }
\fancyhead[RO]{ }
\fancyhead[CO]{\ifthenelse{\value{page}=\fstpage}{\ }{\scriptsize{\headertitle}}}
\fancyhead[LE]{ }
\fancyhead[RE]{ }
\fancyhead[CE]{\scriptsize{Miantao LIU, Shengda LIU, Yu-Zhe LIU}}
\fancyfoot[L]{ }
\fancyfoot[C]{\page}
\fancyfoot[R]{ }
\renewcommand{\headrulewidth}{0.5pt}


\thanks{$^{\ast}$Corresponding author.}


\author{Miantao Liu}
\address{School of Mathematics, Nanjing University, Nanjing, Jiangsu, 210093, China}
\email{dg1921007@smail.nju.edu.cn / miantao.liu@imj-prg.fr (M. Liu);
{\it ORCID}: \href{https://orcid.org/0009-0005-8118-1732}{0009-0005-8118-1732}}

\author{Shengda Liu}
\address{The State Key Laboratory of Multimodal Artificial Intelligence Systems, Institute of Automation,
Chinese Academy of Sciences, Beijing 100190, P. R. China.}
\email{thinksheng@foxmail.com (S. Liu);
{\it ORCID}: \href{https://orcid.org/0000-0003-1382-4212}{0000-0003-1382-4212}}

\author{Yu-Zhe Liu$^*$}
\address{School of Mathematics and Statistics, Guizhou University, Guiyang 550025, Guizhou, P. R. China}
\email{liuyz@gzu.edu.cn / yzliu3@163.com (Y.-Z. Liu);
{\it ORCID}: \href{https://orcid.org/0009-0005-1110-386X}{0009-0005-1110-386X}}

\maketitle
\begin{abstract}
This paper provides a new categorification of the Lebesgue integral with variable upper limits by using normed modules over finite-dimensional $\mathds{k}$-algebras $\itLamb$ and the category $\scrA^p_{\itLamb}$ associated with $\itLamb$. The integration process is redefined through the introduction of an integral partially ordered set and an abstract integral with variable upper limits. Finally, we present two important applications: (1) the categorification of basic elementary functions, including (anti-)trigonometric and logarithmic functions, and (2) a new approach for characterizing the global dimensions of gentle algebras.
\end{abstract}

\setcounter{tocdepth}{1}
\tableofcontents

\def\gldim{\mathrm{gl.dim}}
\def\perm{\mathrm{perm}}
\def\forb{\mathrm{forb}}
\def\rmv{\mathsf{v}}
\def\rmp{\mathsf{w}}

\section{\sectcolor Introduction}

In recent years, category theory has been widely applied in many fields,
and some researchers have focused on the categorifications of differentials \cite{APL2021diff, BCS2015diff, IL2023ana, Lemay2019, Lemay2023} and integrations \cite{CL2018Cart-int, CL2018int, Lei2023FA}.
In \cite{Lei2023FA}, \checkLiu{Tom Leinster provided a new method to describe the Lebesgue integration in terms of category theory, and this categorification differs from that in \cite{CL2018Cart-int} and \cite{CL2018int}.}
He defined a category $\scrA^p$ ($p\ge 1$) over a field $\kk$,
where the objects are triples $(V,v,\delta)$, consisting of a Banach space $V$,
a distinguished element $v \in V$, and a special $\kk$-linear map $\delta$ from \checks{$V\oplus V$} to $V$.
The morphism $f$ between $(V, v, \delta)$ and $(V', v', \delta')$ is
a $\kk$-linear map $f$ from $V$ to $V'$ satisfying $f(v) = v'$ and $\delta' \circ (\checks{f\oplus f}) = f \circ \delta$. \checkLiu{An object $I$ of a category $\mathcal{C}$ is initial if for each object $Z$ of $\mathcal{C}$, there is
exactly one map form $I$ to $Z$ in $\mathcal{C}$ and any two initial objects of a category are isomorphic.}
\checkLiu{In the category $\scrA^1$, Lebesgue integration corresponds to the morphism $\w{T}$ from the initial} object $(L^1([c,d]), \id_{[c,d]}: [c,d] \to 1, \gamma_{\frac{1}{2}})$ to the object $(\kk, 1, m)$. See \cite[Proposition 2.2]{Lei2023FA} for details. \checkLiu{ For Lebesgue integrals with variable upper limit $t$,} let $C_*([c,d])$ be the Banach space of continuous functions $F$ from $[0,1]$ to $\RR$ satisfying $F(c)=0$,
$\mathrm{id}_{[c,d]}$ be the identity function on the
$[c,d]$ and $\eta: C_*([c,d]) \oplus C_*([c,d]) \to C_*([c,d])$
be the $\RR$-linear map defined as
\[\eta(F, G)(x)=
\begin{cases}
\displaystyle\frac{1}{2}F(2x-c), & \text{ if } \displaystyle c\le x\le \frac{c+d}{2}; \\
\displaystyle\frac{1}{2}(F(d)+G(2x-d)), & \text{ if } \displaystyle \frac{c+d}{2} \le x \le d.
\end{cases}\]
The triple $(C_*([c,d]), \mathrm{id}_{[c,d]}, \eta)$ is an object in $\scrA^1$. Let $\int_c^t \cdot \dd\mu$ be the morphism from the initial object to $(C_*([c,d]), \mathrm{id}_{[c,d]}, \eta)$ in $\scrA^1$. This $\RR$-linear map sends each function $f \in L^1([c,d])$ to the Lebesgue integral $(\text{L})\int_{c}^{t} f \dd\mu$ with variable upper limit $t$, see \cite[Proposition 2.4]{Lei2023FA}.

In order to generalize the $\RR$-linear map $\int_c^t \cdot \dd\mu$, we introduced the integral partial order sets and used them to provide a categorification of integrals with variable upper limits and the addition of integrals with variable upper limits.
We represented the Lebesgue integral $(\text{L})\int_{c}^{t} f \dd\mu$ with variable upper limit $t$ as the pair $(t, (\text{L})\int_{c}^{t} f \dd\mu)$  in $\kk\times\kk$,
and treated the pair $([c,t], (\text{L})\int_{c}^{t} f \dd\mu)$ as an element in $\Sigma([c,d])\times\kk$,
where $\Sigma([c,d])$ is the $\sigma$-algebra generated by all subintervals of $[c,d]$).
In Section \ref{sect:iposet}, we introduce \checkLiu{the integral partially ordered sets}
$\w{T}_{\widetilde{\spI}}(f)$ and $\w{T}_{\spI}(f)$ (see Subsection \ref{subsect:iposet}).
Let $\itLamb$ be a finite-dimensional algebra defined over a field $\kk$,
and $\II = [c,d]_{\kk}$ be a fully ordered subset of $\kk$ \checkLiu{whose minimal element is $c$ and whose maximal element is $d$.} Let $\bfS_{\tau}$ be the set of all equivalence classes of elementary simple functions defined on the set
\[ \II_{\itLamb} := \overbrace{[c,d]_{\kk} \times \cdots \times [c,d]_{\kk}}^{\dim_{\kk}\itLamb}, \] and $\II_{\itLamb}$ can be seen as a subset of the finite-dimensional algebra $\itLamb$. Let $\w{\bfS_{\tau}(\II_{\itLamb})}$ be the \checkLiu{completion of the set $\bfS_{\tau}(\II_{\itLamb})$ with respect to the norm topology, see Remark \ref{topology} and Example \ref{completion}~(4).}
Then, for $\alpha \in [c,d]_{\kk}$ and $f\in\w{\bfS_{\tau}(\II_{\itLamb})}$, \checkLiu{we can define $\frakS_{\alpha,f}$ and $\frakS$ as follows:}
\[
\frakS_{\alpha,f} :=
\bigg\{ \bigg(t, (\scrA^p_{\itLamb})\int_{\II_{\itLamb}} f\dd\mu \bigg) \in \kk\times\kk
\ \bigg|\ \alpha\le t\le d \bigg\} \text{\ (see \ref{subsubsect:3.3.1})},
\]
\[
\frakS:=\bigcup_{\alpha \in [c,d]_{\kk} \atop f\in\w{\bfS_{\tau}(\II_{\itLamb})}}
\frakS_{\alpha,f},
\]
In Subsection \ref{subsect:spC}, \checkLiu{ we obtain following maps
\[\spC: \frakS \to \Sigma(\II_{\itLamb})\times\kk, \]
\[\spC^{\natural}: \Sigma(\II_{\itLamb})\times\kk \to \{\II_{\itLamb}\}\times \kk \cong \kk,\]
and
\[+ :  \w{T}_{\widetilde{\spI}}( \w{\bfS_{\tau}(\II_{\itLamb})} )
    \times \w{T}_{\widetilde{\spI}}( \w{\bfS_{\tau}(\II_{\itLamb})} )
    \to  \Sigma(\II_{\itLamb})\times\kk,\]
where $\w{T}_{\widetilde{\spI}}( \w{\bfS_{\tau}(\II_{\itLamb})} ):=
 \bigcup_{
\spI\subseteq\widetilde{\spI},
f\in\w{\bfS_{\tau}(\II_{\itLamb})}
    }
    \w{T}_{\spI}(f)
    $}.
We will prove that the image $\im(+)$ is a left $\itLamb$-module, and the image $\im(\spC)$ of $\spC$ is a left $\itLamb$-submodule of $\im(+)$, see Propositions \ref{lemm:im(+) mod pre} and \ref{prop:spC mod}, respectively.
Furthermore, the above two propositions show the main result of our paper.

\begin{theorem}[{Theorem \ref{thm:spC Lambda-homo}}] \label{mian result:1}
The restriction
\[\spC^{\natural}|_{\im(+)}: \im(+) \to \{\II_{\itLamb}\} \times \kk\]
of $\spC^{\natural}$ is a $\itLamb$-epimorphism.
\end{theorem}

The above theorem provides a description of the addition of two integrals with variable upper limits using the morphism in $\scrA^p_{\itLamb}$.

In Subsection \ref{subsect:spC-special}, we consider some results given in Subsection \ref{subsect:spC} within the framework of the category $\scrA^1$ that satisfies the L-condition (i.e., conditions (L1)--(L6) given in Theorem \ref{thm:LLHZpre-2} (2), cf. \cite[Proposition 2.2]{Lei2023FA}).
We present the following proposition, which utilizes $\scrA^1$ to implement the calculations of Lebesgue integrals with variable upper limits.

\checkLiu{We use the symbol $\biguplus$ to represent a disjoint union.}
\begin{theorem}[{Theorem \ref{thm:proper}}] \label{mian result:2}\rm
Let $\spC_{\alpha}$ be the map
\[  \spC_{\alpha}:
\frakS_{\alpha} :=  \checkLiu{\biguplus_{f\in\Widehat{\bfS_{\mathrm{id}_{\RR}}([c,d])}}} \frakS_{\alpha,f}
\ -\!\!\!\to\  \w{T}_{\spI} := \checkLiu{\biguplus_{f\in\Widehat{\bfS_{\mathrm{id}_{\RR}}([c,d])}}} \w{T}_{\spI}(f)
\]
that sends each element $(t, (\mathrm{L})\int_{\alpha}^{t} f\dd\mu)$ in $\frakS_{\alpha,f}$
to the element $([\alpha,t], \w{T}_{\alpha}^{t}(f))$ in $\spC_{\alpha,f}$
.
If $\scrA^p$ satisfies \checkLiu{the} L-condition, then the following properties hold:
\begin{itemize}
\item[\rm(1)] For any $(\beta, (\mathrm{L})\int_{\alpha}^{\beta} f\dd\mu) \in \frakS_{\alpha, f}$ and any real number $\gamma$ satisfying $\beta\le\gamma\le d$, we have:
\begin{align*}
& \spC^{\natural}\circled\spC
\bigg(\beta, (\mathrm{L})\int_{\alpha}^{\beta} f\dd\mu \bigg)
+ \spC^{\natural}\circled\spC
\bigg(\gamma, (\mathrm{L})\int_{\beta}^{\gamma} f\dd\mu \bigg) \\
=\ & \spC^{\natural}\circled\spC
\bigg(\gamma, (\mathrm{L})\int_{\alpha}^{\gamma} f\dd\mu \bigg).
\end{align*}

\item[\rm(2)]The map $\spC^{\natural}_{\alpha} := \spC^{\natural}|_{\im(\spC_{\alpha})}$
is a $\kk$-linear map. For any $k_1,k_2\in\kk$,
$\tilde{f}:=(\beta, (\mathrm{L})\int_{\alpha}^{\beta} f\dd\mu ) \in \frakS_{\alpha, f}$
and $\tilde{g}:=(\beta, (\mathrm{L})\int_{\alpha}^{\beta} g\dd\mu) \in \frakS_{\alpha, g}$,
we have the following equation:
\[ \spC^{\natural}_{\alpha}
(k_1 \spC_{\alpha}(\tilde{f})+k_2\spC_{\alpha}(\tilde{g}))
=  k_1\cdot \spC^{\natural}_{\alpha} \circled \spC_{\alpha}(\tilde{f})
+k_2\cdot \spC^{\natural}_{\alpha} \circled \spC_{\alpha}(\tilde{g}). \]
\end{itemize}
\end{theorem}
Based on the Theorem \ref{thm:proper}, \checkLiu{we obtain as Corollary \ref{thm:main} the categorification in terms} of our category of the Lebesgue integral with variable upper limits in the sense of Leinster and Meckes. The above results may be applied in the future to other areas of mathematics, such as differential and integral equations. However, in this paper, only two important applications are presented. One is the categorification of elementary functions, including (anti)trigonometric and logarithmic / exponential functions, and the other is a new approach to describe the global dimensions of the gentle algebras, \checkLiu{ we consider two functions. One is defined on gentle algebra $A=\kk\Q/\I$ and the other one is defined on its Koszul dual $A^{!}$ of $A$:}

We consider the \checks{homomorphism} defined on \checks{a} gentle algebras $A=\kk\Q/\I$ and the \checks{homomorphism} defined on the Koszul dual $A^{!}$ of $A$:
\begin{itemize}
\item The \checks{$\Lambda\mbox{-}$homomorphism} $\rmv_{\wp}: A\to \kk$ (the definition of $\rmv_{\wp}$ will be provided in \ref{subsubsect:rmv res} of the Subsection \ref{subsect:rmv}),
where $\wp$ denote some paths on the bound quiver of $A$.

\item The \checks{$\kk\mbox{-}$projection} $\rmp_{\wp}: A^!\to \kk$  (the definition of $\rmp_{\wp}$ will be provided in \ref{subsubsect:LSint} of the Subsection \ref{subsect:gldim Stie int}),
where $\wp$ are some special paths on the bound quiver of $A^!$. 
\end{itemize}

For a gentle algebra $\kk\Q/\I$ \cite{AS1987}, we call a path $P=a_1\cdots a_{l}$ in $\Q$ a permitted thread if
\begin{itemize}
\item[(1)] $a_{i}a_{i+1}\notin\I$, $i\le 1\le l-1$;
\item[(2)] For any arrow $\alpha\in\Q_1$ with $\s(\alpha)=\t(a_{l})$ (resp. $\t(\alpha)=\s(a_1)$),
we have $\alpha a_1\in \I$ (resp. $a_l\alpha\in \I$).
\end{itemize}
Dually, $P=a_1\cdots a_{l}$ is called a forbidden thread if
\begin{itemize}
\item[(1)] $a_{i}a_{i+1}\in\I$, $i\le 1\le l-1$;
\item[(2)] For any arrow $\alpha\in\Q_1$ with $\s(\alpha)=\t(a_{l})$ (resp. $\t(\alpha)=\s(a_1)$),
we have $\alpha a_1\notin \I$ (resp. $a_1\alpha\notin \I$).
\end{itemize}
We denote $\perm(A)$ (resp. $\forb(A)$) by the set of all permitted threads (resp. forbidden thread) on the bound quiver $(\Q,\I)$ of $A$.
Let $\widetilde{l}_F$ (resp. $\widetilde{l}_P$ ) be the number of all elements in the set $\{\s(a_i)\mid 1\le i\le l_F\}$ (resp. the set $\{\s(a_i^{!})\mid 1\le i\le l_P\}$), where $F=a_1\cdots a_{l_F}$ (resp. $P=a_1^{!}\cdots a_{l_P}^{!}$) are the paths in $\forb(A)$ (resp. $\perm(A^!)$), we obtain the following two results:
\begin{theorem}
Let $A$ be a gentle algebra whose global dimension $\gldim A$ is finite and $A^!$ be its Koszul dual, we have:
\begin{itemize}
\item[\rm(1)] {\rm(Theorem \ref{fact:1})}
\[ \frac{\gldim A}{2}
= \sup_{F\in\forb(A)}   \mathop{\int\cdots\int}\limits_{[0,1]^{\times \widetilde{l}_F}} \rmv_F \dd\mu
= \sup_{P\in\perm(A^!)} \mathop{\int\cdots\int}\limits_{[0,1]^{\times \widetilde{l}_P}} \rmv_P \dd\mu,\]
 where $l\in\NN^+$ in $\displaystyle\mathop{\int\cdots\int}\limits_{[0,1]^{\times l}}\cdot \dd\mu$ indicates that the multiple integral is defined on $\kk^l$.

\item[\rm(2)] {\rm(Theorem \ref{fact:2})}
There exists a family of Lebesgue-Stieltjes measures
\[
\checkLiu{
(\varphi_l: [1,t]\mapsto \ln t^l)_{
   \checks{l\in \{l(P)\mid P\in\perm(A^!) \}}
   }
}
~~~~ \text{\rm(see the formula (\ref{formula:fact:2}) )}
\]
defined on the $\sigma$-algebra $\Sigma([1,2])$ such that
\[ \gldim A =  \sup_{P\in\perm(A^!)} (\mathrm{L\text{-}S})\int_1^2 \rmp_P|_{\kk P}\ \dd\varphi_{\checks{l}}, \]
where $\displaystyle(\mathrm{L\text{-}S})\int \cdot \dd\varphi_\checks{l}$ denotes the Lebesgue-Stieltjes integration and $\Sigma([1,2])$ is generated by all subintervals of $[1,2]$.
\end{itemize}
\end{theorem}

\section{\sectcolor Preliminaries} \label{sec:preliminaries}

Let $\kk$ be a field such that the following conditions hold.
\begin{itemize}
  \item[(1)] $\kk$ is equipped with a norm $|\cdot|$.
    That is the function $|\cdot|: \kk \to \RR^{\ge 0}$ such that:
    \begin{itemize}
      \item $|k|=0$ if and only if $k=0$;
      \item $|k_1k_2|=|k_1||k_2|$ holds for all $k_1,k_2\in\kk$;
      \item and the triangle inequality $|k_1+k_2|\le |k_1|+|k_2|$ holds for all $k_1,k_2\in\kk$.
    \end{itemize}
  \item[(2)] The set $\{\mathfrak{B}_r=\{a\in\kk\mid |a|<r\} \mid r\in\RR^+\}$ induces a standard topology $\frakU_{\kk}(0)$ on $\kk$ whose element is called the {\defines  neighborhood} of $0\in\kk$,
  and each element in $\frakU_{\kk}(0)$ is said to be a {\defines neighborhood} of $0$.
  \item[(3)] $\kk$ contains a \checkLiu{totally ordered subset $\II$.}
\end{itemize}

A $\kk$-linear space $\itLamb$ is called a $\kk$-algebra if it is a ring and $k(ab) = (ka)b = a(kb)$ for all $k\in\kk$ and $a, b \in \itLamb$.
Recall that a left $\itLamb$-module is a pair $(V,h)$, where $V$ is a $\kk$-linear space,
and $h$ is a homomorphism
\[h:\itLamb \to \End_{\kk}V, \ a\mapsto (h_a: V \to V)\]
of $\kk$-algebras, $\End_{\kk}V$ is the set of all $\kk$-linear maps of the form $V \to V$.
Equivalently, $h$ induces a left action $\itLamb\times V \to V$, $(a,v)\mapsto av:=h_a(v)$, which satisfies the following conditions:
\begin{itemize}
\item[(1)] $a(v+v')=av+av'$ for any $v,v'\in V$ and $a\in \itLamb$;
\item[(2)] $(a+a')v=av+a'v$ for any $v\in V$ and $a,a'\in \itLamb$;
\item[(3)] $a'(av)=(a'a)v$ for any $v\in V$ and $a,a'\in \itLamb$;
\item[(4)] $1v=v$ for any $v\in V$;
\item[(5)] $(ka)v=k(av)=a(kv)$ for any $v\in V$, $a\in \itLamb$ and $k\in\kk$.
\end{itemize}

\begin{example} \rm
Let $\pmb{M}$ be an $n\times n$ matrix whose elements are real numbers.
Let $\itLamb=\RR[x]$ and $V=\RR^{\oplus n}$, then
\[h: \RR[x] \to \End_{\RR}(\RR^{\oplus n}) \cong
\left(\begin{matrix}
\RR & \RR & \cdots & \RR \\
\RR & \RR & \cdots & \RR \\
\vdots & \vdots & & \vdots \\
\RR & \RR & \cdots & \RR
\end{matrix}\right),\
f(x) \mapsto f(\pmb{M})\]
is a homomorphism which induces the left $\RR[x]$-action
\[ f(x) \pmb{v} := f(\pmb{M}) \pmb{v} \]
such that the conditions (1) -- (5) above hold.
\end{example}

Throughout this paper, we assume that $\dim_{\kk}\itLamb$ is finite $(=n)$ and $B_{\itLamb} = \{b_i\mid 1\le i\le n\}$ is a basis of $\itLamb$ as a $\kk$-linear space. For simplification, $(V,h)$ will be written as $V$, and we assume that $V$ can be infinite-dimensional. We recall some concepts originally derived from Leinster's work in \cite{Lei2023FA}, with further details provided in \cite{LLHZpre}.

\subsection{\sectcolor Normed modules and Banach modules}
We fix a homomorphism $\tau:\itLamb \to \kk$ between two finite-dimensional $\kk$-algebras, where $\kk$ is a field equipped with a norm $|\cdot|: \kk\to\RR^{\ge0}$.

Let $V$ be a vector space over $\kk$, then a norm on $V$ is a map $\Norm{\cdot}{}: V \to \RR^{\ge0}$ satisfying the following four axioms:
\begin{enumerate}
    \item Non-negativity: For every $x \in V$, $\Norm{x}{}\geq 0$.
    \item Positive definiteness: For every $x \in V$, $\Norm{x}{}= 0$ if and only if $x=0.$
    \item Absolute homogeneity: For every $\lambda\in \kk$ and $ x \in V$,
    $\Norm{\lambda x}{} = |\lambda| \Norm{x}{}.$
    \item  Triangle inequality: For every $x\in V$, $y\in V$,
    $\Norm{x+y}{} \leq \Norm{x}{} + \Norm{y}{}.$
\end{enumerate}
The pair $(V, \Norm{\cdot}{})$ is called a normed vector space on $\kk$.

\begin{remark}\label{topology}
For the normed vector space $(V, \Norm{\cdot}{})$, the norm \( \|\cdot\| \) naturally induces a metric and, consequently, induces a topology on \( V \).
This metric between two vectors $x$ and $y$ is given by
\[
d(x, y) = \|x - y\|, \quad \forall x, y \in V.
\]
Using this metric, one can define a topology on \( V \), known as the norm topology or metric topology. The topological basis of this topology consists of all open balls:
\[
B_r(x) = \{ y \in V \mid \|y - x\| < r \}, \quad \forall x \in V, \ r > 0.
\]
This topology determines the open sets in \( V \) and makes \( \|\cdot\| \) continuous and which is compatible with the linear structure of \( V \) in the following ways:
\begin{enumerate}
    \item The addition \( + : V \times V \to V \) is jointly continuous with respect to this topology. This follows directly from the triangle inequality.
    \item The scalar multiplication \( \cdot : \kk \times V \to V \), where \( \kk \) is the underlying scalar field of \( V \), is jointly continuous. This follows from the triangle inequality and homogeneity of the norm.
\end{enumerate}
\end{remark}

\begin{definition} \rm
We define a {\defines normed $\itLamb$-module} as a triple $(V, h, \Norm{\cdot}{})$ consisting of a left $\itLamb$-module $(V,h)$ and a norm $\Norm{\cdot}{}: V \to \RR^{\ge 0}$ on $V$, and for any $a\in \itLamb$ and $m\in V$,
\[\Norm{am}{\tau} = |\tau(a)|\Norm{m}{}.\]
Notice that $\Norm{\cdot}{}$ is written as $\Norm{\cdot}{\tau}$ in our paper since the definition of $\Norm{\cdot}{}$ depends on $\tau$.
\end{definition}

Each normed $\itLamb$-module $(V, h, \Norm{\cdot}{\tau})$ has the canonical induced metric defined by $\Norm{m-n}{\tau}$ for all vectors $m, n$ in $V$.

\begin{definition} \rm
\begin{enumerate}
\item A sequence $m_{1}, m_{2}, \dots, m_{l}$ in $V$ is called the {\defines$\Norm{\cdot}{\tau}$-Cauchy sequence} if, for every $\varepsilon\in\RR^{\ge0}$, there exists an index $N$ such that $\Norm{m_{i}-m_{j}}{\tau}<\varepsilon$ for all $i,j> N$.
\item  A normed $\itLamb$-module $(V, h, \Norm{\cdot}{\tau})$ is called {\defines Banach $\itLamb$-module} if, for every $\Norm{\cdot}{\tau}$-Cauchy sequence $m_{1}, m_{2}, \dots$ in $V$, there exists an object $m$ in $V$, such that
$$\displaystyle \lim_{l \to \infty} \Norm{m_l-m}{\tau} = 0.$$
\end{enumerate}
\end{definition}
For simplicity, each normed or Banach $\itLamb$-module $(V, h, \Norm{\cdot}{\tau})$ \checkLiu{is simply denoted by} $V$.
We denote by $[c,d]_{\kk}$ the interval \checkLiu{in $\kk$} with the minimal element $c$ and the maximal element $d$, \checkLiu{and let $\II$ be the interval $[c,d]_{\kk}$.}
We denote by $\II_{\itLamb}$ a subset of $\itLamb$, which has a one-to-one correspondence with the Cartesian product $\prod_{i=1}^n \II b_i$. We assume that $[c,d]_{\kk}$ contains an element $\xi$ such that $c\prec \xi \prec d$ and there exist two order-preserving bijections $\kappa_{c}: [c,d]_{\kk}\to [c,\xi]_{\kk}$ and $\kappa_{d}: [c,d]_{\kk}\to [\xi,d]_{\kk}$.
For any subset $S$ of $\itLamb$, the function $\id_S$ is defined on $\itLamb$ as follows:\[\id_S: \itLamb \to \kk,
x \mapsto \begin{cases}
1, & \text{ if } x\in S; \\
0, & \text{otherwise}.
\end{cases}\]
We will provide an algorithm required for our categorification.
\subsection*{The algorithm:}
\begin{enumerate}
    \item First step: we fix an element  $\xi_{11}\in\II=[c,d]_{\kk}$, then the element $\xi_{11}$ divides $\II$ (we also use $\II^{(01)}$ to represent $\II$) to two subsets defied as follows:
 $$\II^{(11)}:=[c,\xi_{11}]_{\kk},~~~\II^{(12)}:=[\xi_{11},d]_{\kk}.$$  Then $\kappa_{c}(d)=\xi_{11}=\kappa_{d}(c)$ and this step is equal to
$$\begin{tikzpicture}
    \draw[thick] (-0.5,0) -- (12.5,0);

    \fill (0, 0) circle (2pt);
    \node[above] at (0, 0) {$c$};
\fill[red] (4, 0) circle (2pt);
\node[red][above] at (4, 0) { $\xi_{11}$};
\fill (12, 0) circle (2pt);
    \node[above] at (12, 0) { $d$};
\end{tikzpicture}$$

    \item Second step: let $\xi_{22}=\xi_{11}, c=\xi_{20}, d= \xi_{24}$ and $\xi_{21}$ and $\xi_{23}$ are defined as two elements in $\II$ such that
 \begin{itemize}
  \item $c\prec\xi_{21}=\kappa_c\kappa_c(d) = \kappa_c\kappa_d(c) = \kappa_c(\xi_{11})\prec\xi_{22}$;
  \item $\xi_{22}\prec\xi_{23}=\kappa_d\kappa_d(c)=\kappa_d\kappa_c(d) = \kappa_d({\xi_{11}})\prec d$.
\end{itemize}
Thus, $\II$ is divided into four subsets of the form $\II^{(2t)} = [\xi_{2t}, \xi_{2\ t+1}]_{\kk}$ ($0\le t\le 3$),
with $c=\xi_{20} \prec \xi_{21} \prec \xi_{22} \prec \xi_{23} \prec \xi_{24}=d$. This step is equal to
$$\begin{tikzpicture}
    \draw[thick] (-0.5,0) -- (12.5,0);

    \fill (0, 0) circle (2pt);
    \node[above] at (0, 0) {$c=\xi_{20}$};

    \fill (12, 0) circle (2pt);
    \node[above] at (12, 0) {$\xi_{24} = d$};

    \fill[red] (4, 0) circle (2pt);
    \node[red][above] at (4, 0) {$\xi_{11}=\xi_{22}$};

    \fill[blue] (1.33, 0) circle (2pt);
    \node[blue][above] at (1.33, 0) {$\xi_{21}$};

    \fill[blue] (6.66, 0) circle (2pt);
    \node[blue][above] at (6.66, 0) {$\xi_{23}$};

\end{tikzpicture}$$

\item Following step: repeating the above step $t$ times, we obtain a segmentation
$c=\xi_{t0} \prec \xi_{t1} \prec \xi_{t2} \prec \cdots \prec \xi_{t2^t}=d$
of $[c,d]_{\kk}$, where each $[\xi_{ts}, \xi_{t\ s+1}]_{\kk}$ is written as $\II^{(t\ s+1)}$. Let $\II^{(01)} = [\xi_{00}, \xi_{01}]_{\kk} = [c,d]_{\kk}$ in our paper.

\item Finally, we obtain $2^t$ order-preserving bijections $\kappa_{ts}: \II^{(01)} = [c,d]_{\kk} \to \II^{(t\ s+1)}$
which are compositions of several order-preserving bijections of $\{\kappa_c, \kappa_d\}$, see \Pic \ref{fig:kappa} and Example \ref{exp:kappa}.
\end{enumerate}

\begin{figure}[htbp]
\centering
\includegraphics[width=15cm]{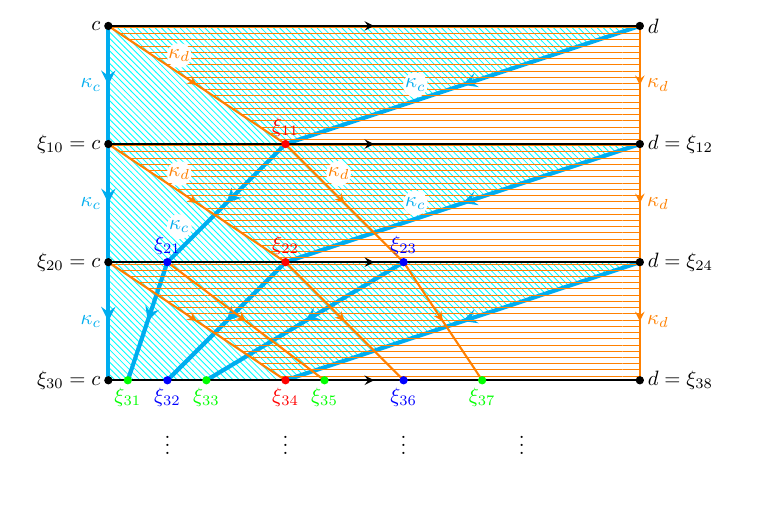}
\caption{Order-preserving bijections $\kappa_c$ and $\kappa_d$}
\label{fig:kappa}
\end{figure}

\begin{example} \rm \rm \label{exp:kappa} 
Consider $\kappa_c$ and $\kappa_d$ and the intervals given in \Pic \ref{fig:kappa} for $t=3$.
\begin{align*}
  & \II^{(31)}=[\xi_{30},\xi_{31}]_{\kk},
 && \II^{(32)}=[\xi_{31},\xi_{32}]_{\kk}, \\
  & \II^{(33)}=[\xi_{32},\xi_{33}]_{\kk},
 && \II^{(34)}=[\xi_{33},\xi_{34}]_{\kk}, \\
  & \II^{(35)}=[\xi_{34},\xi_{35}]_{\kk},
 && \II^{(36)}=[\xi_{35},\xi_{36}]_{\kk}, \\
  & \II^{(37)}=[\xi_{36},\xi_{37}]_{\kk},
 && \II^{(38)}=[\xi_{37},\xi_{38}]_{\kk},
\end{align*}
Then we have $2^t=2^3=8$ bijections as follows:
\begin{align*}
 & (1)\ \kappa_{31} = \kappa_c\kappa_c\kappa_c;
&& (2)\ \kappa_{32} = \kappa_c\kappa_c\kappa_d; \\
 & (3)\ \kappa_{33} = \kappa_c\kappa_d\kappa_c;
&& (4)\ \kappa_{34} = \kappa_c\kappa_d\kappa_d; \\
 & (5)\ \kappa_{35} = \kappa_d\kappa_c\kappa_c;
&& (6)\ \kappa_{36} = \kappa_d\kappa_c\kappa_d; \\
 & (7)\ \kappa_{37} = \kappa_d\kappa_d\kappa_c;
&& (8)\ \kappa_{38} = \kappa_d\kappa_d\kappa_d.
\end{align*}
\end{example}

For any family of subsets $(\II^{(x_iy_i)})_{1\le i\le n}$ ($1\le y_i\le 2^{x_i}$), we denote the function $\id_{(x_iy_i)_i}$ by $\id_{\II_{\itLamb}}\big|{}_{\prod_{i=1}^{n}\II^{(x_iy_i)}}$ for simplicity, where $\II^{(x_iy_i)} \cong \II^{(x_iy_i)}\times\{b_i\}\subseteq \II_{\itLamb}$ holds for all $i$.

\begin{definition} \rm
We define {\defines step function spaces} $E_u$ on $\II_{\itLamb}${\rm} as the set of all functions of the form
\[\sum_{(x_iy_i)_i} k_{(x_iy_i)_i}\id_{(x_iy_i)_i} \ (= \sum_i k_i\id_{I_i} \text{ for simplification}), \]
where $k_{(x_iy_i)_i}\in\kk$, and $(x_iy_i)_i$ corresponds to the Cartesian product $\prod_{i=1}^{n}\II^{(x_iy_i)}$.

$E_u$ is a left $\itLamb$-module defined as the pair $(E_u, h_{E_u}: \itLamb\to\End_{\kk}E_u)$,
where $h_{E_u}$ is given by \[ h_{E_u}: a \mapsto \big((h_{E_u})_a: f \mapsto \tau(a)f\big), \]
and for any $a, a'\in \itLamb$ and $f\in E_u$, we have \[(h_{E_u}(aa'))(f) = (h_{E_u})_{aa'}(f) = \tau(aa')f = \tau(a)\tau(a')f = ((h_{E_u})_a\circled(h_{E_u})_{a'})(f).\]
\end{definition}

\begin{definition} \rm
A map $f:\II_{\itLamb}\to\kk$ defined by $f(k_1, \ldots, k_n) = \sum_{i=1}^t k_i\id_{X_i}$ is called an {\defines elementary simple function} on $\II_{\itLamb}$, where $k_i\in\kk$.
\end{definition}

\checks{A $\sigma$-algebra $A$ defined on a nonempty set $X$ is a subset of the power set $P(X)$ of $X$ such that $X$ is an element of $A$, and $A$ is closed under the operations of taking complements, countable union, and countable intersections.}

\begin{example} \rm
If $X$ is any nonempty set, then the largest $\sigma$-algebra on $X$ is the power set $\mathcal{P}(X)$ consisting of all subsets of $X$. The smallest $\sigma$-algebra is $\{\emptyset, X\}$.
\end{example}

\begin{definition} \rm
If $\mathcal{F}$ is any collection of subsets of $X$, then the smallest $\sigma$-algebra containing $\mathcal{F}$ is called {\defines the $\sigma$-algebra generated by $\mathcal{F}$}. This $\sigma$-algebra is the intersection of all $\sigma$-algebras that contain $\mathcal{F}$.
\end{definition}
 Let $X_{i}$ be a subset of $\II$ in the forms $(c_{ij},d_{ij})_{\kk}$,  $(c_{ij},d_{ij}]_{\kk}$, $[c_{ij},d_{ij})_{\kk}$ and $[c_{ij},d_{ij}]_{\kk}$, where $c\preceq c_{ij} \prec d_{ij} \preceq d$. We assume $\bigcup_{i=1}^t X_i = \II_{\itLamb}$, and $X_i\cap X_j=\varnothing$ for all $1\le i\ne j\le t$. We define $\Sigma(\II)$ as the $\sigma$-algebra generated by all subsets of the following form,
\begin{center}
$\{(c',d')_{\kk}, [c',d')_{\kk}, (c',d']_{\kk}, [c',d']_{\kk} \mid c\preceq c'\preceq d'\preceq d\}$,
\end{center}
and we fix a measure $\mu_{\II}$ on $\Sigma(\II)$ such that $\mu_{\II}(\{a\})$ is zero for any $a\in\II$. We call $\Sigma(\II_{\itLamb})$ the $\sigma$-algebra generated by all subsets $\sum_{i=1}^n I_ib_i$ in $\itLamb$, where $I_{i}~(1\le i\le n)$ is one of the intervals $(c_i,d_i)_{\kk}$, $(c_i,d_i]_{\kk}$, $[c_i,d_i)_{\kk}$ and $[c_i,d_i]_{\kk}$ with $c\preceq c_i \preceq d_i \preceq d$. Then $\mu_{\II}$ has an induced measure $\mu_{\II_{\itLamb}}$ (simply denoted $\mu$) on $\Sigma(\II_{\itLamb})$, which is defined by
\begin{align} 
  \mu_{\II_{\itLamb}}\bigg(\sum_{i=1}^n I_ib_i\bigg) = \prod_{i=1}^n \mu_{\II}(I_i). \nonumber
\end{align}
We can define the equivalence class of elementary simple functions as follows:

\begin{definition} \rm
Two elementary simple functions $f$ and $g$ are said to be {\defines equivalent} if and only if the measure of $\{x \in \II_{\itLamb} \mid f(x)-g(x)\ne 0\}$ is zero.
\end{definition}

\begin{example} \rm\label{completion}
Let $\bfS_{\tau}(\II_{\itLamb})$ be the set of all equivalence classes of elementary simple functions.
\begin{enumerate}
\item Consider the set $\bfS_{\tau}(\II_{\itLamb})$ as the pair $(\bfS_{\tau}(\II_{\itLamb}), h_{\bfS_{\tau}(\II_{\itLamb})})$, it
forms a left $\itLamb$-module with the map $h_{\bfS_{\tau}(\II_{\itLamb})}$
\[ h_{\bfS_{\tau}(\II_{\itLamb})} : a\mapsto \big( (h_{\bfS_{\tau}(\II_{\itLamb})})_a: f\mapsto \tau(a)f \big).  \] for any $a\in\itLamb$ and $f\in\bfS_{\tau}(\II_{\itLamb})$.
    \item Consider $\bfS_{\tau}(\II_{\itLamb})$ as the triple $(\bfS_{\tau}(\II_{\itLamb}), h_{\bfS_{\tau}(\II_{\itLamb})}, \Norm{\cdot}{})$, it is a normed $\itLamb$-module with the norm given by
\[ \Norm{f}{} = \bigg(\sum_{i} |k_i|^p\mu(X_i)^p\bigg)^{\frac{1}{p}}, \]
where $f=\sum_{i}k_i\id_{X_i}$.
\item Since $E_u \subseteq \bfS_{\tau}(\II_{\itLamb})$, $E_u=(E_u, h_{E_u}, \Norm{\cdot}{})$ is also a normed $\itLamb$-module with the norm
\[ \Norm{f}{} = \bigg(\sum_{i} |k_i|^p\mu(I_i)^p\bigg)^{\frac{1}{p}}. \]

\item Let \( \w{\bfS_{\tau}(\II_{\itLamb})} \) be the completion of \( \bfS_{\tau}(\II_{\itLamb}) \) with respect to the topology induced by the norm, and it becomes a complete space in this topology.
Then, $\w{\bfS_{\tau}(\II_{\itLamb})}$ is a Banach $\itLamb$-module. It is straightforward to observe that
\[\kk \cong E_{0} \subseteq E_{1} \subseteq \ldots \subseteq E_{t} \subseteq \ldots
\subseteq \bfS_{\tau}(\II_{\itLamb}) \subseteq \w{\bfS_{\tau}(\II_{\itLamb})}. \]
The authors have checked that $\underrightarrow{\lim}E_u \cong \w{\bfS_{\tau}(\II_{\itLamb})}$ \cite[Lemma 5.4]{LLHZpre}, or cf. \cite[Examples 2.2.4 (h) and 2.2.6 (g)]{Bor1994}.
\end{enumerate}
\end{example}

\begin{remark}
If $\itLamb$ is either $\RR$ or $\mathbb{C}$, then $\II_{\itLamb}$ is an interval $[c,d]\subseteq\RR$. In this case, Leinster \cite{Lei2023FA} shows $$\underrightarrow{\lim}E_u \cong L^1([c,d]).$$
\end{remark}

\subsection{\sectcolor The categories $\scrN^p_{\itLamb}$ and $\scrA^p_{\itLamb}$ and their special objects} \label{subsec:cats}
In this section, we recall the  categories $\scrN^p$, the category $\scrA^p$ ($p\ge 1$), and also some results in \cite{Lei2023FA} and \cite{LLHZpre}.

Assume that $\itLamb$ is an arbitrary finite-dimensional $\kk$-algebra with a basis $B_{\itLamb}=\{b_1, \cdots, b_n\}$.
For any $2^n$ $\itLamb$-modules $X_i$, $\ldots$, $X_{2^n}$,
\[\bigoplus\limits_{i=1}^{2^n}{}_p \  X_i = X_1\oplus_p \cdots \oplus_p X_{2^n}\]
is the direct sum $\bigoplus\limits_{i=1}^{2^n}X_i$ with the norm $\Vert\cdot\Vert$ defined by
\[\Vert (x_1,\cdots, x_{2^n}) \Vert =
\bigg(
\bigg(
\frac{\mu_{\II}(\II)}{\mu(\II_\checkLiu{\itLamb})}
\bigg)^n
\sum_{i=1}^{2^n} \Vert x_i \Vert^p
\bigg)^{\frac{1}{p}}.\]

We define the integral normed module category and the integral Banach module category, both of which are fundamental to the study of normed modules over algebras.
\begin{definition} \rm
(1) The category $\scrN^p_{\itLamb}$ (denoted by $\scrN^p$ for convenience) is defined as {\defines the category of integral normed modules of $\itLamb$} as follows.
\begin{itemize}
\item The object of $\scrN^p$ is the triple $(N, v, \delta)$, where
\begin{itemize}
\item $N$ is a normed $\itLamb$-module,
\item $v$ is an element of $N$ satisfying $\Vert v\Vert \le \mu(\II_{\itLamb})$,
\item $\delta: N^{\oplus_p 2^n} \to N$ is a $\itLamb$-homomorphism which maps $(v,\cdots, v)$ to $v$. For any monotone decreasing Cauchy sequence $\{x_i\}_{i\in\NN}$ in $N^{\oplus_p 2^n}$ (where the partial order ``$\preceq$'' is defined as follows: $x',x''\in N$, $x'\preceq x''$ if and only if $\Vert x'-\underleftarrow{\lim}x_i \Vert \le \Vert x''-\underleftarrow{\lim}x_i \Vert$), the homomorphism $\delta$ commutes with the inverse limit, i.e., $\delta(\underleftarrow{\lim} x_i) = \underleftarrow{\lim} \delta(x_i)$.
\end{itemize}
\item For any two objects $(N, v, \delta)$ and $(N',v',\delta')$, the morphism $(N, v, \delta) \to (N',v',\delta')$ is a $\itLamb$-homomorphism $f: N\to N$ such that $f(v)=v'$ and $f\circled\delta = \delta'\circled f^{\oplus 2^n}$.
\end{itemize}

(2) The category $\scrA^p_{\itLamb}$ (we will write it as $\scrA^p$ for simplicity)
       is the \checkLiu{full} subcategory of those objects where $N$ is a Banach $\itLamb$-module, we denote by $\scrA^p_{\itLamb}$ {\defines the category of integral Banach modules of $\itLamb$}.
\end{definition}

Next, we provide an important example for an objects lying in $\scrN^p$ (resp., $\scrA^p$).
We recall juxtaposition map in this subsection.

\begin{definition} \rm
A {\defines juxtaposition map} $\gamma_{\xi}: \bfS_{\tau}(\II)^{\oplus_{p} 2^n} \to \bfS_{\tau}(\II)$ is a $\kk$-linear map defined as follows,
  \[\gamma_{\xi}(\pmb{f}) (k_1,\ldots,k_n)
= \sum_{(\delta_1, \ldots, \delta_n)\in \{c,d\}\times\cdots\times\{c,d\}}
\id_{\kappa_{\delta_1}(\II)\times \cdots\times
\kappa_{\delta_n}(\II)}\cdot f_{\delta_1, \ldots, \delta_n} (\kappa_{\delta_1}^{-1}(k_1),\ldots,\kappa_{\delta_n}^{-1}(k_n)),\]
where $k_1\ne\xi$, $\ldots$, $k_n\ne\xi$.
\end{definition}

\begin{example} \rm
Let $\itLamb = \kk b_1 + \kk b_2$ be a finite-dimensional algebra over a field $\kk$ whose dimension is $2$,
and let $[c,d]_{\kk}$ ($\subseteq \kk$) be an interval. Let $\II_{\itLamb}=[c,d]_{\kk}b_1 + [c,d]_{\kk}b_2$ and let $$f_{c,c}, f_{c,d}, f_{d,c}, f_{d,d}:  \II_{\itLamb} \to \kk$$ be four functions.
Then $\gamma_{\xi}: \bfS_{\tau}(\II)^{\oplus_{p} 4} \to \bfS_{\tau}(\II)$ sends the quadruple
\[ \pmb{f}(k_1,k_2) = (f_{c,c}(k_1,k_2), f_{c,d}(k_1,k_2), f_{d,c}(k_1,k_2), f_{d,d}(k_1,k_2))\] to the following function
\[ \tilde{f}_{c,c}(k_1,k_2) + \tilde{f}_{c,d}(k_1,k_2) + \tilde{f}_{d,c}(k_1,k_2) + \tilde{f}_{d,d}(k_1,k_2),  \]
where
\begin{align*}
    \tilde{f}_{c,c}(k_1,k_2)
& = \id_{(c,\xi)\times (c,\xi)} f_{c,c}(\kappa_c^{-1}(k_1),\kappa_c^{-1}(k_2)), \\
    \tilde{f}_{c,d}(k_1,k_2)
& = \id_{(c,\xi)\times (\xi,d)} f_{c,d}(\kappa_c^{-1}(k_1),\kappa_d^{-1}(k_2)), \\
    \tilde{f}_{d,c}(k_1,k_2)
& = \id_{(\xi,d)\times (c,\xi)} f_{d,c}(\kappa_c^{-1}(k_1),\kappa_d^{-1}(k_2)), \\
    \tilde{f}_{d,d}(k_1,k_2)
& = \id_{(\xi,d)\times (\xi,d)} f_{d,d}(\kappa_c^{-1}(k_1),\kappa_d^{-1}(k_2)),
\end{align*}
they are shown in \Pic \ref{fig:jux.map}.
\begin{figure}[htbp]
\begin{center}
\definecolor{pistachio}{rgb}{0.25,1,0.75}
\begin{tikzpicture}[scale=1.2]
\draw [pistachio][line width=8pt][shift={(0, 3)}] (-3,0) -- (3,0);
\draw [pistachio][line width=8pt][shift={(0,-3)}] (-3,0) -- (3,0);
\draw [pistachio][line width=8pt][shift={( 3,0)}] (0,-3) -- (0,3);
\draw [pistachio][line width=8pt][shift={(-3,0)}] (0,-3) -- (0,3);
\fill [pistachio][shift={(-3,-3)}]
     (-1.25,-1.25) -- ( 1.25,-1.25) -- ( 1.25, 1.25) -- (-1.25, 1.25) -- (-1.25,-1.25);
\fill [pistachio][shift={( 3,-3)}]
     (-1.25,-1.25) -- ( 1.25,-1.25) -- ( 1.25, 1.25) -- (-1.25, 1.25) -- (-1.25,-1.25);
\fill [pistachio][shift={( 3, 3)}]
     (-1.25,-1.25) -- ( 1.25,-1.25) -- ( 1.25, 1.25) -- (-1.25, 1.25) -- (-1.25,-1.25);
\fill [pistachio][shift={(-3, 3)}]
     (-1.25,-1.25) -- ( 1.25,-1.25) -- ( 1.25, 1.25) -- (-1.25, 1.25) -- (-1.25,-1.25);
\fill [pistachio][shift={(-3,-3)}]
     (-1,-0.3) -- (-1,-1) -- (-2.5,-0.66)
     node[left]{$\bfS(\II_{\itLamb})^{\oplus 4}$};
\fill [pink]
     (-1.25,-1.25) -- ( 1.25,-1.25) -- ( 1.25, 1.25) -- (-1.25, 1.25) -- (-1.25,-1.25);
\fill [pink] (1,-0.3) -- (1,-1) -- (5,-0.66)
     node[right]{$\bfS(\II_A)$};
\fill[left color=   red!37, right color=    white]
  (-1  ,-1  ) -- (-0.3,-1  ) -- (-0.3,-0.3) -- (-1  ,-0.3);
\fill[ top color=  blue!37,bottom color=    white]
  (-0.3,-1  ) -- ( 1  ,-1  ) -- ( 1  ,-0.3) -- (-0.3,-0.3);
\fill[ top color=    white,bottom color=orange!37]
  (-0.3,-0.3) -- ( 1  ,-0.3) -- ( 1  , 1  ) -- (-0.3, 1  );
\fill[left color=    white, right color=violet!37]
  (-0.3,-0.3) -- (-1  ,-0.3) -- (-1  , 1  ) -- (-0.3, 1  );
\draw [black][line width=1pt]
     (-1,-1) node[ left]{$c$} node[below]{$c$}
  -- ( 1,-1) node[right]{$d$}
  -- ( 1, 1)
  -- (-1, 1) node[above]{$d$}
  -- (-1,-1);
\fill [shift={(-3, 3)}][left color=    white, right color=violet!37]
     (-1,-1) -- ( 1,-1) -- ( 1, 1) -- (-1, 1) -- (-1,-1);
\draw [black][line width=1pt][shift={(-3, 3)}]
     (-1,-1) node[ left]{$c$} node[below]{$c$}
  -- ( 1,-1) node[right]{$d$}
  -- ( 1, 1)
  -- (-1, 1) node[above]{$d$}
  -- (-1,-1);
\fill [shift={( 3, 3)}][ top color=    white,bottom color=orange!37]
     (-1,-1) -- ( 1,-1) -- ( 1, 1) -- (-1, 1) -- (-1,-1);
\draw [black][line width=1pt][shift={( 3, 3)}]
     (-1,-1) node[ left]{$c$} node[below]{$c$}
  -- ( 1,-1) node[right]{$d$}
  -- ( 1, 1)
  -- (-1, 1) node[above]{$d$}
  -- (-1,-1);
\fill[shift={( 3,-3)}][ top color=  blue!37,bottom color=    white]
     (-1,-1) -- ( 1,-1) -- ( 1, 1) -- (-1, 1) -- (-1,-1);
\draw [black][line width=1pt][shift={( 3,-3)}]
     (-1,-1) node[ left]{$c$} node[below]{$c$}
  -- ( 1,-1) node[right]{$d$}
  -- ( 1, 1)
  -- (-1, 1) node[above]{$d$}
  -- (-1,-1);
\fill [shift={(-3,-3)}][left color=   red!37, right color=    white]
     (-1,-1) -- ( 1,-1) -- ( 1, 1) -- (-1, 1) -- (-1,-1);
\draw [black][line width=1pt][shift={(-3,-3)}]
     (-1,-1) node[ left]{$c$} node[below]{$c$}
  -- ( 1,-1) node[right]{$d$}
  -- ( 1, 1)
  -- (-1, 1) node[above]{$d$}
  -- (-1,-1);
\draw[shift={( 0, 3)}][<->] (-1.8,0)--(1.8,0);
\draw[shift={( 0, 3)}]      (0,0) node[above]{$\oplus$};
\draw[shift={( 0,-3)}][<->] (-1.8,0)--(1.8,0);
\draw[shift={( 0,-3)}]      (0,0) node[below]{$\oplus$};
\draw[shift={(-3, 0)}][<->] (0,-1.8)--(0,1.8);
\draw[shift={(-3, 0)}]      (0,0) node[ left]{$\oplus$};
\draw[shift={( 3, 0)}][<->] (0,-1.8)--(0,1.8);
\draw[shift={( 3, 0)}]      (0,0) node[right]{$\oplus$};
\draw[white][line width=1pt] (-0.3,-1)--(-0.3, 1);
\draw[white][line width=1pt] (-1,-0.3)--( 1,-0.3);
\draw[cyan]
  (-0.3,-1) node{$\bullet$} node[below]{$\xi$}
  (-1,-0.3) node{$\bullet$} node[ left]{$\xi$};
\draw[red] (-2,-2) -- (-0.3,-0.3);
\draw[red] (-2.5,-2.0) node[above]{$\kappa_c$};
\draw[red] (-2.0,-2.5) node[right]{$\kappa_c$};
\draw[red][postaction={on each segment={mid arrow = red}}] (-4,-2) -- (-1  ,-0.3);
\draw[red][postaction={on each segment={mid arrow = red}}] (-2,-4) -- (-0.3,-1  );
\draw[red][shift={(-3,-3)}] (0,0) node{$f_{(c,c)}$};
\draw[red][shift={(-0.6,-0.6)}] (0,0) node{\scriptsize$\tilde{f}_{(c,c)}$};
\draw[blue] ( 2,-2) -- (-0.3,-0.3);
\draw[blue] ( 2.5,-2.0) node[above]{$\kappa_d$};
\draw[blue] ( 2.0,-2.5) node[ left]{$\kappa_c$};
\draw[blue][postaction={on each segment={mid arrow = blue}}] ( 4,-2) -- ( 1  ,-0.3);
\draw[blue][postaction={on each segment={mid arrow = blue}}] ( 2,-4) -- (-0.3,-1  );
\draw[blue][shift={( 3,-3)}] (0,0) node{$f_{(d,c)}$};
\draw[blue][shift={( 0.3,-0.6)}] (0,0) node{\scriptsize$\tilde{f}_{(d,c)}$};
\draw[orange] ( 2, 2) -- (-0.3,-0.3);
\draw[orange] ( 2.5, 2.0) node[below]{$\kappa_d$};
\draw[orange] ( 2.0, 2.5) node[ left]{$\kappa_d$};
\draw[orange][postaction={on each segment={mid arrow = orange}}] ( 4, 2) -- ( 1  ,-0.3);
\draw[orange][postaction={on each segment={mid arrow = orange}}] ( 2, 4) -- (-0.3, 1  );
\draw[orange][shift={( 3, 3)}] (0,0) node{$f_{(d,d)}$};
\draw[orange][shift={( 0.3, 0.3)}] (0,0) node{\scriptsize$\tilde{f}_{(d,d)}$};
\draw[violet] (-2, 2) -- (-0.3,-0.3);
\draw[violet] (-2.5, 2.0) node[below]{$\kappa_c$};
\draw[violet] (-2.0, 2.5) node[right]{$\kappa_d$};
\draw[violet][postaction={on each segment={mid arrow = violet}}] (-4, 2) -- (-1  ,-0.3);
\draw[violet][postaction={on each segment={mid arrow = violet}}] (-2, 4) -- (-0.3, 1  );
\draw[violet][shift={(-3, 3)}] (0,0) node{$f_{(c,d)}$};
\draw[violet][shift={(-0.6, 0.3)}] (0,0) node{\scriptsize$\tilde{f}_{(c,d)}$};
\end{tikzpicture}
\caption{Juxtaposition map (in the case for $\dim_{\kk}\itLamb = 2$)}
\label{fig:jux.map}
\end{center}
\end{figure}
One can check that the triple $(\bfS_{\tau}(\II_{\itLamb}), \id_{\II_{\itLamb}}, \gamma_{\xi})$ is an object in $\scrN^p$,
and its completion $(\w{\bfS_{\tau}(\II_{\itLamb})}, \id_{\II_{\itLamb}}, \w{\gamma}_{\xi})$ is an object in $\scrA^p$, see \cite{LLHZpre}.
\end{example}

The following result is important for $\scrA^p$.

\begin{theorem} \label{thm:LLHZpre-1}{\rm(\!\cite[Theorems 6.3]{LLHZpre})}
$(\w{\bfS_{\tau}(\II_{\itLamb})}, \id_{\II_{\itLamb}}, \w{\gamma}_{\xi})$ is an initial object in $\scrA^p$, that is, for any object $(N,v,\delta)$ in $\scrA^p$, there is a unique morphism
\[\w{T}_{(N,v,\delta)} \in \Hom_{\scrA^p}((\w{\bfS_{\tau}(\II_{\itLamb})}, \id_{\II_{\itLamb}}, \w{\gamma}_{\xi}), (N,v,\delta)). \]
\end{theorem}

\subsection{\sectcolor The categorification of Lebesgue integrations}

The following result shows that if $\kk$ is a complete field containing $\RR$, then the morphism $T_{(N,v,\delta)}$ in Theorem \ref{thm:LLHZpre-1} provides a categorification of Lebesgue integrations in the case where $(N,v,\delta)=(\kk,\mu(\II_{\itLamb}),m)$, and $m$ is a $\kk$-linear map of the form $\kk^{\oplus 2} \to \kk$.

\begin{theorem} \label{thm:LLHZpre-2}
{\rm(1)} {\rm(\!\!\cite[Theorems 7.6]{LLHZpre})}
If $\kk$ is a complete field containing $\RR$, then there is an object in $\scrA^p$ of the form $(\kk,\mu(\II_{\itLamb}),m)$ such that
\[ \w{T}_{(\kk,\mu(\II_{\itLamb}),m)}: f \mapsto \w{T}_{(\kk,\mu(\II_{\itLamb}),m)}(f)
\text{ {\rm(}denote it by } \ (\scrA^p)\int_{\II_{\itLamb}} \cdot\dd\mu)\]
is a unique $\itLamb$-homomorphism in $\Hom_{\scrA^p}((\bfS_{\tau}(\II_{\itLamb}),\id_{\II_{\itLamb}}, \gamma_{\xi}), (\kk,1,m))$ and satisfies the following conditions:
\begin{itemize}
\item  $\w{T}_{(\kk,\mu(\II_{\itLamb}),m)}(\id_{\II_{\itLamb}}) = \mu(\II_{\itLamb})$;
\item  $\w{T}_{(\kk,\mu(\II_{\itLamb}),m)}: \bfS_{\tau}(\II_{\itLamb}) \to \kk$ is a homomorphism of $\kk$-modules;
\item  $\w{T}_{(\kk,\mu(\II_{\itLamb}),m)}(|f|) \le |\w{T}_{(\kk,\mu(\II_{\itLamb}),m)}(f)|$.
\end{itemize}

{\rm(2)} {\rm(Leinster)}
If $\scrA^p$ satisfies {\defines L-conditions}, that is,
\begin{enumerate}
\item[{\rm(L1)}] $p=1$;
\item[{\rm(L2)}] $\kk=\RR=\itLamb$ $($in this case, the norms of $\kk$ and $\itLamb$ coincide$)$, or $\kk=\mathbb{C}=\itLamb$ $($in this case, the norm of $\mathbb{C}$ is defined by the modulus of complex numbers$)$;
\item[{\rm(L3)}] $\II=[x_1,x_2]$ $(=\II_{\RR})$;
\item[{\rm(L4)}] $\displaystyle\xi=\frac{x_1+x_2}{2}$, $\displaystyle\kappa_{x_1}(x)=\frac{x+x_1}{2}$, $\displaystyle\kappa_{x_2}(x)=\frac{x+x_2}{2}$;
\item[{\rm(L5)}] $\tau=\mathrm{id}_{\kk}$;
\item[{\rm(L6)}] $\mu: \Sigma(\II_{\itLamb}) \to \RR^{\ge 0}$ is the Lebesgue measure defined on the $\sigma$-algebra $\Sigma(\II_{\itLamb})$,
\end{enumerate}
\checks{then} there exists an object of the form $(\kk,\mu([x_1,x_2]),m)$ in $\scrA^1$ such that
$\w{T}_{(\kk,\mu([x_1,x_2]),m)}: \w{\bfS_{\tau}(\II_{\itLamb})} \to \kk$ is a $\kk$-linear map
sending any $f\in\w{\bfS_{\tau}(\II_{\itLamb})}$ to its Lebesgue integral:
\[ \w{T}_{(\kk,\mu([x_1,x_2]),m)}(f) = \mathrm{(L)}{\int_{x_1}^{x_2}} f\dd\mu. \]
\end{theorem}

\begin{remark}
The definitions of the $\itLamb$-homomorphisms $m$ in different categories $\scrA^p$ are different; see \cite[Section 7]{LLHZpre} and \cite[Section 2]{Lei2023FA}. However, in this paper, we will not consider the detailed computation of $m$.
\end{remark}

\section{\sectcolor Integral partial ordered sets} \label{sect:iposet}


\subsection{\sectcolor The definition of integral partial ordered set} \label{subsect:iposet}

We fix a finite-dimensional algebra $\itLamb$ with a basis $B_{\itLamb}=\{b_1, \ldots, b_n\}$ and we fix a homomorphism $\tau: \itLamb \to \kk$ of the $\kk$-algebras. Based on the facts in Subsection~\ref{subsec:cats}, we can see that $\scrN^p$ and $\scrA^p$ are determined by the interval $\II=[c,d]_{\kk}$ ($\subseteq\kk$), the maps $\kappa_c: [c,d]_{\kk} \to [c,\xi]_{\kk}$, $ \kappa_d: [c,d]_{\kk} \to [\xi,d]_{\kk}$
and the measure $\mu$ defined on $\Sigma(\II_{\itLamb})$ (or equivalently, the measure $\mu$ defined on $\Sigma(\II)$).
That is, $\scrN^p$ and $\scrA^p$ can be viewed as the quintuples $(\scrN^p_{\itLamb}, \II, \kappa_c, \kappa_d, \mu)$ and $(\scrA^p_{\itLamb}, \II, \kappa_c, \kappa_d, \mu)$, respectively.

\begin{definition} \rm
$\II$ is said to satisfy the {\defines segmentation condition} if it satisfies the following, for any subinterval $[c',d']_{\kk}$ of $\II$, there is an element $\xi'$ such that $c'\prec \xi' \prec d'$ and there are two order-preserving bijections $\kappa_{c'}: [c',d']_{\kk}\to [c',\xi']_{\kk}$, $\kappa_{d'}: [c',d']_{\kk}\to [\xi,d']_{\kk}$ defined on the subintervals.
\end{definition}
We will assume that $\II$ satisfies the segmentation condition. For any subinterval $[c',d']_{\kk}$ of $\II$ and two order-preserving bijections $\kappa_{c'}$ and $\kappa_{d'}$,
we can define $[c',d']_{\itLamb} = \sum_{i=1}^n [c',d']_{\kk}b_i$. In this case, we obtain two categories $(\scrN^p_{\itLamb}, [c',d']_{\kk}, \kappa_{c'}, \kappa_{d'}, \mu|_{\Sigma([c',d']_{\itLamb})})$
and $(\scrA^p_{\itLamb}, [c',d']_{\kk}, \kappa_{c'}, \kappa_{d'}, \mu|_{\Sigma([c',d']_{\itLamb})})$, where the $\sigma$-algebra $\Sigma([c',d']_{\itLamb})$ is a $\sigma$-subalgebra of $\Sigma(\II_{\itLamb})$. For simplification, we do not distinguish $\mu|_{\Sigma([c',d']_{\itLamb})}$ and $\mu$.

Let $\spI$ be a partial ordered subset of $\widetilde{\spI}=\{[t_1,t_2]_{\kk} \mid c\preceq t_1 \preceq t_2 \preceq d \}$ with the partial ordered of $\widetilde{\spI}$ defined by ``$\subseteq$''. We denote by $[t_1,t_2]_{\kk}$ an element of a partial ordered set $\spI$.
\begin{definition} \rm
An {\defines integral partial ordered set} $\w{T}_{\spI}(f)$ of $f\in\w{\bfS_{\tau}(\II_{\itLamb})}$ (\iposet of $f$ in short) is defined as a poset whose elements are pairs of the form $([t_1, t_2]_{\kk}, \w{T}_{t_1}^{t_2}(f)),$
where
\begin{itemize}
  \item[(1)] $\w{T}_{t_1}^{t_2}:=\w{T}_{(\kk, \mu([t_1,t_2]_{\itLamb}), m)}$ is the $\itLamb$-homomorphism in the category $(\scrA^p_{\itLamb}, [t_1,t_2]_{\kk},$ $\kappa_{t_1}, \kappa_{t_2}, \mu)$ from $(\Widehat{\bfS_{\tau}([t_1,t_2]_{\itLamb})}, \id, \gamma)$ to $(\kk, \mu([t_1, t_2]_{\itLamb}), m)$;

\item[(2)] and $[t_1, t_2]_{\itLamb}$ runs through all the elements of $\spI$.
\end{itemize}
\end{definition}

The partial ordered ``$\preceq$'' of $\w{T}_{\spI}(f)$ is induced by $\spI$, that is, the following statements are equivalent:
\begin{itemize}
  \item[(1)] $([s_1, s_2]_{\kk}, \w{T}_{s_1}^{s_2}(f)) \preceq ([t_1, t_2]_{\kk}, \w{T}_{t_1}^{t_2}(f))$;
  \item[(2)] $[s_1, s_2]_{\kk} \subseteq [t_1, t_2]_{\kk}$;
  \item[(3)] The fact that $\w{\bfS_\tau(\II_\Lambda)}$ is the initial object in the category implies that \checkLiu{$\w{T}_{s_1}^{s_2}(f)$ is the restriction of the $\w{T}_{t_1}^{t_2}(f)$}.
  \item[(4)] The fact that $\bfS_\tau(\II_\Lambda)$ is the initial object in the category implies that $T_{s_1}^{s_2}(f)$ is the restriction of the $T_{t_1}^{t_2}(f)$.
\end{itemize}

\begin{example} \rm
\begin{enumerate}
\item Given a function $f\in\bfS_{\tau}(\II_{\itLamb})$, if $\spI$ is an ascending chain $[t_{11}, t_{21}]_{\kk} \subset [t_{12}, t_{22}]_{\kk} \subset [t_{13}, t_{23}]_{\kk} \subset \cdots$ on $\widetilde{\spI}$, then $\w{T}_{\spI}(f) = (\{ ([t_{1i}, t_{2i}]_{\kk}, \w{T}_{t_{1i}}^{t_{2i}}(f)) \mid i\in\NN^+ \}, \preceq). $
 \item Given a function $f\in\bfS_{\tau}(\II_{\itLamb})$, if $\spI = \{[c,t]_{\kk}\mid c\preceq t\preceq d\}$, then the \checkLiu{cardinality} of $\spI$ and  $\w{T}_{\spI}(f)=(\{ ([c,t]_{\kk})$ and $\w{T}_{c}^{t} \mid c\preceq t\preceq d\}, \preceq)$ coincide.
\end{enumerate}
\end{example}

\subsection{\sectcolor An addition induced by integral partial ordered set} \label{subsect:iposet's addi}
For any two elements
\begin{center}
$   ([u,v]_{\kk}, \w{T}_{u}^{v}(f)), ([s,t]_{\kk}, \w{T}_{s}^{t}(g))
\in \w{T}_{\widetilde{\spI}}( \w{\bfS_{\tau}(\II_{\itLamb})} )$
$:=\displaystyle
\checkLiu{
\biguplus
\limits_{
    \spI\subseteq\widetilde{\spI},
    f\in\w{\bfS_{\tau}(\II_{\itLamb})}
  }
}
\w{T}_{\spI}(f)$,
\end{center}
where the symbol $\biguplus$ represents a disjoint union.
The addition functor
\begin{align}\label{formula:addi-map}
+:  \w{T}_{\widetilde{\spI}}( \w{\bfS_{\tau}(\II_{\itLamb})} )
\times \w{T}_{\widetilde{\spI}}( \w{\bfS_{\tau}(\II_{\itLamb})} )
\to  \Sigma(\II_{\itLamb})\times\kk
\end{align}
is defined as follows:
\begin{align}
& \ ([u,v]_{\kk}, \w{T}_{u}^{v}(f)) + ([s,t]_{\kk}, \w{T}_{s}^{t}(g)) \nonumber
\\ :=\ & ( U, \ \w{T}_{\min\{u,v,s,t\}}^{\max\{u,v,s,t\}}
(\id_{[u,v]_{\itLamb}}f+\id_{[s,t]_{\itLamb}}g))  \label{formula:addi T}
\\  =\ &
{\begin{cases}
(U, \w{T}_{u}^{t}(\id_{[u,v]_{\itLamb}}f+\id_{[s,t]_{\itLamb}}g)),
& \text{ if } u\preceq v \prec s\preceq t; \ \spadesuit \\ 
(U, \w{T}_{s}^{v}(\id_{[s,t]_{\itLamb}}g+\id_{[u,v]_{\itLamb}}f)),
& \text{ if } s\preceq t \prec u\preceq v; \ \clubsuit \\ 
(U, \w{T}_{u}^{t}(\id_{[u,s]_{\itLamb}}f + \id_{[s,v]_{\itLamb}}(f+g) + \id_{[v,t]_{\itLamb}}g)),
& \text{ if } u\preceq s \prec v\preceq t; \\ 
(U, \w{T}_{s}^{v}(\id_{[s,u]_{\itLamb}}g + \id_{[u,t]_{\itLamb}}(f+g) + \id_{[t,v]_{\itLamb}}f)),
& \text{ if } s\preceq u \prec t\preceq v; \\ 
(U, \w{T}_{s}^{t}(g+\id_{[u,v]_{\kk}}f)),
& \text{ if } [u,v]_{\kk}\subseteq [s,t]_{\kk}; \ \diamo \\ 
(U, \w{T}_{u}^{v}(f+\id_{[s,t]_{\kk}}g)),
& \text{ if } [s,t]_{\kk}\subseteq [u,v]_{\kk}; \ \heart \\ 
\end{cases}} \label{formula:addi}
\end{align}
where $U=[\min\{u,v,s,t\},\max\{u,v,s,t\}]_{\kk}$.

\begin{lemma} \label{lemm:spcase of addi}
Take $f=g$ in (\ref{formula:addi}).
\begin{itemize}
\item[{\rm(1)}] If $u\preceq v=s\preceq t$, then $U=[u,t]_{\kk}$, and in this case, we have
\[ ([u,v]_{\kk}, \w{T}_{u}^{v}(f)) + ([s,t]_{\kk}, \w{T}_{s}^{t}(f)) =  ([u,t]_{\kk}, \w{T}_{u}^{t}(f)). \]

\item[{\rm(2)}] If $s\preceq t=u\preceq v$, then $U=[s,v]_{\kk}$, and in this case, we have
\[ ([u,v]_{\kk}, \w{T}_{u}^{v}(f)) + ([s,t]_{\kk}, \w{T}_{s}^{t}(f)) =  ([s,v]_{\kk}, \w{T}_{u}^{t}(f)). \]

\item[{\rm(3)}] If $[u,v]_{\kk}=[s,t]_{\kk}$, then $U=[u,v]_{\kk}=[s,t]_{\kk}$, and in this case, we have
\[ ([u,v]_{\kk}, \w{T}_{u}^{v}(f)) + ([s,t]_{\kk}, \w{T}_{s}^{t}(f)) =  ([u,v]_{\kk}, \w{T}_{u}^{v}(2f)). \]
\end{itemize}
\end{lemma}

\begin{proof}
Statements (1) and (2) are direct corollaries of (\ref{formula:addi}) $\spadesuit$, respectively.
The statement (3) can be obtain by $\diamo$ or $\heart$ as follows:
\begin{align*}
& ([u,v]_{\kk}, \w{T}_{u}^{v}(f)) + ([s,t]_{\kk}, \w{T}_{s}^{t}(f)) \\
= \ & (U, \w{T}_{s}^{t}(f+\id_{[u,v]_{\kk}}f)) \\
= \ & (U, \w{T}_{s}^{t}(\id_{U}f+\id_{U}f)) \\
= \ & (U, \w{T}_{s}^{t}(2f)) \\
= \ & (U, \w{T}_{u}^{v}(2f)).
\end{align*}
\end{proof}

\begin{lemma} \label{lemm:im(+) linear}
The image of the map $+$ given by (\ref{formula:addi-map}) is a $\kk$-linear space, where the $\kk$-action $$\kk\times \im(+) \to \im(+)$$ is defined by $k\cdot(S,r) := (S,kr)$.
\end{lemma}

\begin{proof}
One can check that $\im(+)$ is an Abelian group with zero element $(\varnothing,0)$.
Let $k, k_1, k_2 \in \kk$ and $(S,r), (S_1,r_1), (S_2,r_2) \in \im(+)$. Firstly, we will show that $\im(+)$ is a $\kk$-linear space. In order to show $k\cdot(S,r)\in \im(+)$ for all $k\in\kk$, we can assume $(S,r) = ( [y_1, y_2]_{\kk}, \ \w{T}_{y_1}^{y_2}
(\id_{[u,v]_{\itLamb}}f+\id_{[s,t]_{\itLamb}}g))$
by (\ref{formula:addi T}), where $y_1=\min\{u,v,s,t\}$ and $y_2=\max\{u,v,s,t\}$. Then
\begin{align}\label{formula:im(+) linear}
k(S,r) = (S,kr) = ([y_1, y_2]_{\kk}, \ k\cdot \w{T}_{y_1}^{y_2}
(\id_{[u,v]_{\itLamb}}f+\id_{[s,t]_{\itLamb}}g)).
\end{align}
Notice that $\w{T}_{y_1}^{y_2} = \w{T}_{(\kk, \mu([y_1,y_2]_{\itLamb}), m)}$ is a $\itLamb$-homomorphism
in the category $(\scrA^p_{\itLamb}, [y_1,y_2]_{\itLamb}, \kappa_{y_1},$ $\kappa_{y_2}, \mu)$,
and $\w{\bfS_{\tau}(\II_{\itLamb})}$ is a left $\itLamb$-module,
thus
\[ k\cdot \w{T}_{y_1}^{y_2} (\id_{[u,v]_{\itLamb}}f+\id_{[s,t]_{\itLamb}}g)
= \w{T}_{y_1}^{y_2} (\id_{[u,v]_{\itLamb}}kf+\id_{[s,t]_{\itLamb}}kg). \]
Then the right of (\ref{formula:im(+) linear}) equals the sum
$([u,v]_{\kk}, \w{T}_{u}^{v}(kf)) + ([s,t]_{\kk}, \w{T}_{s}^{t}(kg)),$
it is an element in $\im(+)$.

Secondly, we have the following equations.
\begin{itemize}
\item[(a)] $1\cdot(S,r) = (S, 1\cdot r) = (S, r)$, where $1$ is the identity element of $\kk$.
\item[(b)] $(k_1k_2)\cdot (S,r) = (S, (k_1k_2)r) = (S, k_1(k_2r)) = k_1(S, k_2r) = k_1\cdot(k_2\cdot (S,r))$.
\item[(c)] $(k_1+k_2)\cdot (S,r) = k_1(S,r) + k_2(S,r)$.
\item[(d)] $k((S_1,r_1)+(S_2,r_2)) = k(S_1,r_1)+k(S_2,r_2)$.
\end{itemize}
The proofs of (a) and (b) are straightforward. We will show (c); the proof of (d) is similar to (c).

\begin{align*}
\text{The right of (c)}=\ \ \ & k_1(S,r) + k_2(S,r) = (S,k_1r)+(S,k_2r) \\
\mathop{=\!=}^{(\ref{formula:im(+) linear})} \
& ([y_1, y_2]_{\kk}, \ k_1\cdot \w{T}_{y_1}^{y_2}
(\id_{[u,v]_{\itLamb}}f+\id_{[s,t]_{\itLamb}}g)) \\
& + ([y_1, y_2]_{\kk}, \ k_2\cdot \w{T}_{y_1}^{y_2}
(\id_{[u,v]_{\itLamb}}f+\id_{[s,t]_{\itLamb}}g)) \\
=\ \ & ([y_1, y_2]_{\kk}, \ \w{T}_{y_1}^{y_2}
(k_1(\id_{[u,v]_{\itLamb}}f+\id_{[s,t]_{\itLamb}}g))) \\
& + ([y_1, y_2]_{\kk}, \ \w{T}_{y_1}^{y_2}
(k_2(\id_{[u,v]_{\itLamb}}f+\id_{[s,t]_{\itLamb}}g))) \\
\mathop{=\!=}^{(\ref{formula:addi T})} \
& ([y_1, y_2]_{\kk}, \ \w{T}_{y_1}^{y_2}
((k_1+k_2)(\id_{[u,v]_{\itLamb}}f+\id_{[s,t]_{\itLamb}}g))) \\
=\ \ & ([y_1, y_2]_{\kk}, \ (k_1+k_2)\w{T}_{y_1}^{y_2}
(\id_{[u,v]_{\itLamb}}f+\id_{[s,t]_{\itLamb}}g)) \\
=\ \ & ([y_1, y_2]_{\kk}, (k_1+k_2)r) \\
=\ \ & (S, (k_1+k_2)r) = (k_1+k_2)\cdot (S,r) = \text{the left of (c)}.
\end{align*}

\end{proof}

\begin{lemma} \label{lemm:im(+) mod pre}
For any $a\in \itLamb$ and $(S,r)\in \im(+)$, we have $\tau(a)(S,r) \in \im(+)$.
\end{lemma}

\begin{proof}
By the $\kk$-linear space structure of $\im(+)$, we have $\tau(a)(S,r) = (S,\tau(a)r) \in \im(+)$ by Lemma \ref{lemm:im(+) linear}.
\end{proof}

\begin{proposition} \label{prop:im(+) mod}
The image $\im(+)$ is a left $\itLamb$-module whose left $\itLamb$-action is defined as
\[ \itLamb \times \im(+) \to \im(+),\]
\begin{center}
$(a,(S,r)) \mapsto a\star(S,r):=(S,\tau(a)r).$
\end{center}
\end{proposition}

\begin{proof}
By Lemmas \ref{lemm:im(+) linear} and \ref{lemm:im(+) mod pre}, we only need to prove that the following equations for all $k\in\kk$, $a, a_1, a_2 \in \itLamb$ and $(S,r), (S_1,r_1), (S_2,r_2) \in \im(+)$.
\begin{itemize}
\item[(a')] $1\star(S,r) = (S, \tau(1)r) = (S, r)$,
where $1$ is the identity element of $\itLamb$, and $\tau(1)$ is the identity element of $\kk$.
\item[(b')] $(a_1a_2)\star(S,r) = a_1\star(a_2\star(S,r))$.
%

\item[(c')] $(a_1+a_2)\star (S,r) = a_1\star(S,r) + a_2\star(S,r)$.

%

\item[(d')]  $a\star((S_1,r_1)+(S_2,r_2)) = a\star(S_1,r_1)+a\star(S_2,r_2)$.

%
\item[(e')] $(ka)\star (S,r) \mathop{=}\limits^{\spadesuit} a\star(k(S,r))
\mathop{=}\limits^{\clubsuit} k(a\star(S,r))$.
\end{itemize}
We only prove (d') and (e'), and the proofs of (a'), (b'), and (c') are similar to (d').

By the definition of left $\itLamb$-action, we have
\begin{align*}
a \star((S_1,r_1)+(S_2,r_2))
=\ & \tau(a)((S_1,r_1)+(S_2,r_2)) \\
\mathop{=}\limits^{\heart}\
& \tau(a)(S_1,r_1) + \tau(a)(S_2,r_2) \\
=\ & a\star(S_1,r_1) + a\star (S_2,r_2),
\end{align*}
where $\heart$ is obtained by Lemma \ref{lemm:im(+) linear}.

Since $\tau$ is a homomorphism between two finite-dimensional $\kk$-algebras $\itLamb$ and $\kk$, it is clear that $\tau(ka)=\tau(k)\tau(a)=k\tau(1)\tau(a)=k\tau(a)$.
Thus, we have
\[(ka)\star (S,r) = k \tau(a) (S,r) = \tau(a)k (S,r) \mathop{=}\limits^{\diamo} \tau(a)(r(S,k)) = a\star(r(S,k)),\]
where $\diamo$ is obtained by Lemma \ref{lemm:im(+) linear}. Therefore, $\spadesuit$ holds. We can prove $\clubsuit$ in a similar way.
\end{proof}

\subsection{\sectcolor Map $\spC$ in the sense of the category $\scrA^p$} \label{subsect:spC}
For a variable $t$, the pair $([\alpha,t]_{\kk}, \w{T}_{\alpha}^{t}(f))$ induces a function $[\alpha, t] \mapsto \w{T}_{\alpha}^{t}(f)$ whose variable is an interval variable.
Each pair $([\alpha,t]_{\kk}, \w{T}_{\alpha}^{t}(f))$ corresponds to a pair $(t, (\scrA^p)\int_{[\alpha,t]_{\itLamb}} f\dd\mu)$ in $\kk\times\kk$. The pair $(t, (\scrA^p)\int_{[\alpha,t]_{\itLamb}} f\dd\mu)$ induces a \checkLiu{function} $t\mapsto (\scrA^p)\int_{[\alpha,t]_{\itLamb}} f\dd\mu$ with a numeric variable.
In this subsection, we will consider this correspondence, see (\ref{mapsto:spC}), which will help us categorify the integral with variable upper limit.

\subsubsection{\sectcolor Map $\spC$} \label{subsubsect:3.3.1}
For any $f\in\w{\bfS_{\tau}(\II_{\itLamb})}$ and $\alpha\in (c,d)_{\kk} := \{ t \in \II \mid c\preceq t\preceq d\}$,
let $\spI=\{[\alpha,t]_{\itLamb} = \sum_{i=1}^n[\alpha,t]b_i \mid \alpha\le t\le d\}$, we obtain
\[
\w{T}_{\spI}(f)
= \{ ([\alpha,t]_{\kk}, \w{T}_{\alpha}^{t}(f))
\mid \alpha\le t\le d
\}
\subseteq \Sigma(\II_{\itLamb})\times\kk.
\]
For any $f\in\Widehat{\bfS_{\tau}([c,d])}$ and $\alpha \in (c,d)_{\kk}$, we can define the following map
\[ \spC_{\alpha}: \frakS_{\alpha,f} :=
\bigg\{
\bigg(t, (\scrA^p)\int_{[\alpha,t]_{\itLamb}} f\dd\mu \bigg)
\in \kk\times \kk
\ \bigg| \
\alpha\le t\le d
\bigg\}
\longrightarrow
\w{T}_{\spI}(f) \]
by
\begin{align}\label{mapsto:spC}
\bigg(t, (\scrA^p)\int_{[\alpha,t]_{\itLamb}} f\dd\mu \bigg)
\mapsto
([\alpha,t]_{\kk}, \w{T}_{\alpha}^{t}(f)).
\end{align}
Furthermore, the above definition induces a map $\spC$ from
$ \bigcup_{\alpha \in (c,d)_{\kk} \atop f\in \w{\bfS_{\tau}(\II_{\itLamb})}}
\frakS_{\alpha, f}
= \frakS_{(c,d)_{\kk}, \w{\bfS_{\tau}(\II_{\itLamb})}} $
($=\frakS$ when there is no confusion) to $\Sigma(\II_{\itLamb})\times\kk$.

\subsubsection{\sectcolor Surjection $\spC^{\natural}\circ\spC$ }


For any subset $I\in\Sigma(\II_{\itLamb})$ of $\II_{\itLamb}$,
the following lemma shows that each integral $(\scrA^p) \int_{I} f\dd\mu$ on $I$
can be viewed as an integral on $\II_{\itLamb}$.

\begin{lemma} \label{lemm:spC surj}
Let $\spC^{\natural}$ be the map $\Sigma(\II_{\itLamb})\times\kk \to \{\II_{\itLamb}\}\times\kk$ defined by $(S, k) \to (\II_{\itLamb}, k)$ for any $S\in\Sigma(\II_{\itLamb})$ and $k\in\kk$.
\begin{itemize}
\item[{\rm(1)}]
The composition
\[ \spC^{\natural} \circled \spC: \ \
\xymatrix@C=1.5cm{ \frakS  \ar[r]^{\spC}
& \Sigma(\II_{\itLamb})\times\kk \ar[r]^{\spC^{\natural}}
& \{\II_{\itLamb}\}\times\kk} \]
is surjective.

\item[{\rm(2)}]
Furthermore, for any $(S, k)\in \Sigma(\II_{\itLamb})\times\kk$,
there is an element $(t, (\scrA^p)\int_{\II_{\itLamb}} f \dd\mu )$ in $\frakS$ such that
\[  \spC^{\natural} \circled \spC \bigg(t, (\scrA^p)\displaystyle\int_{\II_{\itLamb}} f \dd\mu \bigg)
= (\II_{\itLamb}, k)
= \spC^{\natural}(S,k). \]
\end{itemize}
\end{lemma}

\begin{proof}
For any pair $(S, k)$ in $\Sigma(\II_{\itLamb})\times\kk$, let
\begin{center}
$f = \displaystyle \frac{k}{\mu(S)}\cdot\id_{S} : \II_{\itLamb} \to \kk$,
$x \mapsto
\begin{cases}
1, & \text{ if } x \in S; \\
0, & \text{ if } x \in \II_{\itLamb}\backslash S,
\end{cases}$
\end{center}
then we obtain
\[ \w{T}_{(\kk, \mu(\II_{\itLamb}), m)}(f)
= (\scrA^p)\int_{\II_{\itLamb}} f \dd\mu
= \frac{k}{\mu(S)} \cdot (\scrA^p)\int_{\II_{\itLamb}} \id_{S} \dd\mu
= \frac{k}{\mu(S)} \cdot \mu(S)
= k.
\]
Thus, there is an element
\[
([c,d]_{\kk}, \w{T}_{c}^{d}(f))
= \bigg( [c,d]_{\kk} , (\scrA^p)\int_{\II_{\itLamb}} f \dd\mu \bigg)
= ([c,d]_{\kk}, k)  \]
in $\frakS$, it may not be equal to $(S,k)$.
It is easy to see the following facts
\begin{itemize}
\item[(a)] $\{\II_{\itLamb}\}\times\kk$ is a subset of $\Sigma(\II_{\itLamb})\times\kk$ since $\II_{\itLamb} = [c,d]_{\itLamb} = \sum\limits_{i=1}^n [c,d]_{\itLamb}b_i$ is measurable;
\item[(b)] $\spC^{\natural}(S,k) = ([c,d]_{\kk}, k)$;
\item[(c)] $([c,d]_{\kk}, k)$ is an element in $\Sigma(\II_{\itLamb})\times\kk$ by (a) and we have
\[ \spC\bigg(d, (\scrA^p)\int_{\II_{\itLamb}} f \dd\mu \bigg)
= ([c,d]_{\kk}, k). \]
\end{itemize}
Therefore, we obtain that
\[ \spC^{\natural}\circled\spC \bigg(d, (\scrA^p)\int_{\II_{\itLamb}} f \dd\mu \bigg)
=  \spC^{\natural}([c,d]_{\kk}, k)
=  (\II_{\kk}, k) = \spC^{\natural}(S,k),\]
We finish the proof of (2).

Due to (a), elements in $\{\II_{\itLamb}\}\times\kk$ can be viewed as elements in $\Sigma(\II_{\itLamb})\times\kk$, then based on the above proof, we get $\spC^{\natural} \circled \spC$ is surjection, we complete the proof of (1).
\end{proof}

\begin{example} \rm
Take $\itLamb=\RR$ and $\scrA^p_{\itLamb}=\scrA^1_{\RR}$ satisfying L-condition (assume $\II=[c,d]$ and $\xi=\frac{1}{2}$ in this case).
We provide an instance for Lemma \ref{lemm:spC surj}.
Consider Lebesgue integral of $f\in\Widehat{\bfS_{\mathrm{id}_{\RR}}([c,d])}$,
then for any $(t, (\scrA^1)\int_{[c,t]} f\dd\mu) \in \frakS_{c, f} \subseteq \frakS$,
we have \[ \spC^{\natural}\circled\spC\bigg(t, (\scrA^1)\int_{[c,t]} f\dd\mu\bigg)
= \spC^{\natural}\bigg([c,t], (\mathrm{L})\int_c^t f\dd\mu\bigg)
= \bigg([c,d], (\mathrm{L})\int_c^t f\dd\mu\bigg). \]
Correspondingly, three pairs as above provide an equation as follows,
\[ (\scrA^1)\int_{[c,t]} f\dd\mu = (\mathrm{L})\int_c^t f\dd\mu = (\mathrm{L})\int_c^d \id_{[c,t]}f \dd\mu. \]
The middle integral can be viewed as an integral with variable upper limit.
\end{example}

\begin{lemma} \label{lemm:spC inj}
The map $\spC$ is an injection.
\end{lemma}

\begin{proof}
Let $(s, (\scrA^p)\int_{[\alpha,s]_{\itLamb}} f\dd\mu ) \in \frakS_{\alpha,f}$
and $(t, (\scrA^p)\int_{[\beta, t]_{\itLamb}} g\dd\mu ) \in \frakS_{\beta,g}$
be two elements in $\frakS$, such that
\[
\spC\bigg(s, (\scrA^p)\displaystyle\int_{[\alpha,s]_{\itLamb}} f\dd\mu \bigg)
= \spC\bigg(t, (\scrA^p)\displaystyle\int_{[\beta, t]_{\itLamb}} g\dd\mu \bigg).
\]
By the definition of $\spC$, we obtain the following equation
\[ ([\alpha,s]_{\kk}, \w{T}_{\alpha}^{s}(f)) = ([\beta,t]_{\kk}, \w{T}_{\beta}^{t}(g)) \]
in $\Sigma(\II_{\itLamb})\times\kk$.
Thus, $[\alpha,s]_{\kk}=[\beta,t]_{\kk}$ and $\w{T}_{\alpha}^{s}(f)=\w{T}_{\beta}^{t}(g)$.
It follows that $\alpha=\beta$, $s=t$, and
\[
(\scrA^p)\int_{[\alpha,s]_{\itLamb}} f \dd\mu
= \w{T}_{\alpha}^{s}(f)
= \w{T}_{\beta}^{t}(g)
= (\scrA^p)\int_{[\beta, t]_{\itLamb}} g \dd\mu  \]

\end{proof}

The following lemma is parallel to Lemma \ref{lemm:im(+) linear}.

\begin{lemma} \label{lemm:spC linear}
$\im(\spC)$ is a $\kk$-linear subspace of $\im(+)$.
\end{lemma}

\begin{proof}
Recall that $\im(+)$ is a $\kk$-linear space, see Lemma \ref{lemm:im(+) linear}, any element of $\im(+)$ is of the form
\[( [y_1,y_2]_{\itLamb}, \ \w{T}_{y_1}^{y_2}
(\id_{[u,v]_{\itLamb}}f+\id_{[s,t]_{\itLamb}}g) \dd\mu ), \]
where $y_1=\min\{u,v,s,t\}$ and $y_2=\max\{u,v,s,t\}$,
then for any $(t, (\scrA^p)\int_{[\alpha,t]_{\itLamb}} f\dd\mu )\in \frakS$, we have
\begin{align*}
\spC\bigg(t, (\scrA^p)\displaystyle\int_{[\alpha,t]_{\itLamb}} f\dd\mu \bigg)
= ([\alpha,t]_{\kk}, \w{T}_{\alpha}^{t}(f))
= ([\alpha,t]_{\kk}, \w{T}_{\alpha}^{t}
(\id_{[\alpha,t]_{\itLamb}}f + \id_{[\alpha,t]_{\itLamb}} 0) ),
\end{align*}
where $u=s=\alpha$, $v=t$, $\min\{u,v,s,t\}=\alpha$, $\max\{u,v,s,t\}=t$, and $U=[\alpha,t]_{\kk}$.
Thus, $\im(\spC)\subseteq \im(+)$.

Assume $k_1,k_2\in\kk$ and  $(s, (\scrA^p)\int_{[\alpha,s]_{\itLamb}} f\dd\mu )$, $(t, (\scrA^p)\int_{[\beta,t]_{\itLamb}} g\dd\mu )$ $\in \frakS$,
we have
\begin{align*}
&  k_1 \spC\bigg(s, (\scrA^p)\int_{[\alpha,s]_{\itLamb}} f\dd\mu \bigg)
+ k_2 \spC\bigg(t, (\scrA^p)\int_{[\beta, t]_{\itLamb}} g\dd\mu \bigg) \\
\mathop{=}\limits^{\spadesuit}\ &
k_1([\alpha,s]_{\kk}, \w{T}_{\alpha}^{s}(f))
+ k_2([\beta, t]_{\kk}, \w{T}_{\beta }^{t}(g)) \\
\mathop{=}\limits^{\clubsuit}\ &
([\alpha,s]_{\kk}, k_1\w{T}_{\alpha}^{s}(f))
+ ([\beta, t]_{\kk}, k_2\w{T}_{\beta }^{t}(g)) \\
\mathop{=}\limits^{\heart}\ &
([\alpha,s]_{\kk}, \w{T}_{\alpha}^{s}(k_1f))
+ ([\beta, t]_{\kk}, \w{T}_{\beta }^{t}(k_2g)) \\
\mathop{=}\limits^{\diamo}\ &
([y_1,y_2]_{\kk}, \
\w{T}_{y_1}^{y_2}
(k_1\id_{[\alpha,s]_{\itLamb}} f + k_2 \id_{[\beta, t]_{\itLamb}} g)
) \\
\in\ & \im(+),
\end{align*}
where:
\begin{itemize}
\item[(a)] $y_1=\min\{\alpha,\beta,s,t\}$, $y_2=\max\{\alpha,\beta,s,t\}$;
\item[(b)] $\spadesuit$ is given by (\ref{mapsto:spC});
\item[(c)] $\clubsuit$ is given by the $\kk$-action in Lemma \ref{lemm:im(+) linear};
\item[(d)] $\heart$ holds since $\w{T}_{\alpha}^{s}$ and $\w{T}_{\beta}^{t}$ are $\itLamb$-homomorphisms in
the categories $(\scrA^p_{\itLamb}, [\alpha,s]_{\kk},$ $\kappa_{\alpha}, \kappa_{s}, \mu)$
and $(\scrA^p_{\itLamb}, [\beta,t]_{\kk}, \kappa_{\beta}, \kappa_{t}, \mu)$, respectively
(see Theorem \ref{thm:LLHZpre-2});
\item[(e)] and $\diamo$ is given by the definition of the addition $+$ in (\ref{formula:addi-map}).
\end{itemize}
Thus, $\im(\spC)$ is a $\kk$-linear subspace of $\im(+)$.
\end{proof}

The following lemma is parallel to Lemma \ref{lemm:im(+) mod pre}.

\begin{lemma} \label{lemm:spC mod pre}
For any $a\in\itLamb$ and $(t, (\scrA^p)\int_{[\alpha,t]_{\itLamb}} f\dd\mu )\in \frakS$, we have
\[ \tau(a)\spC\bigg(t, (\scrA^p)\int_{[\alpha,t]_{\itLamb}} f\dd\mu \bigg) \in \im(\spC). \]
\end{lemma}

\begin{proof}
We have
\begin{align*}
& \tau(a)\spC\bigg(t, (\scrA^p)\int_{[\alpha,t]_{\itLamb}} f\dd\mu \bigg) \\
=\ & \tau(a)([\alpha,t]_{\kk},\w{T}_{\alpha}^{t}(f)) \\
=\ & ([\alpha,t]_{\kk},\tau(a)\w{T}_{\alpha}^{t}(f)) \\
=\ & ([\alpha,t]_{\kk},\w{T}_{\alpha}^{t}(\tau(a)f)) \\
=\ & \spC\bigg(t, (\scrA^p)\int_{[\alpha,t]_{\itLamb}} \tau(a)f\dd\mu \bigg)
\in \im(\spC)
\end{align*}
\end{proof}

Then, by Lemmas \ref{lemm:spC linear} and \ref{lemm:spC mod pre}, we obtain the following proposition, which is parallel to Proposition \ref{prop:im(+) mod}.

\begin{proposition} \label{prop:spC mod}
The image $\im(\spC)$ is a left $\itLamb$-submodule of $\im(+)$.
\end{proposition}

\begin{proof}
By Lemma \ref{lemm:spC linear}, we know that $\im(\spC)$ is a $\kk$-linear space.
By Lemma \ref{lemm:spC mod pre}, we can check that $\im(\spC)$ is also a left $\itLamb$-module.
Lemma \ref{lemm:spC linear} shows that $\im(\spC)$ is a $\kk$-linear subspace of $\im(+)$. \checkLiu{The $\Lambda$-module structure of $\im(\spC)$ is the restriction of the structrue on $\im(+)$}. Thus, we finish the proof.
\end{proof}

We can define the addition \checkLiu{in} $\{\II_{\itLamb}\} \times \kk$ by $(\II_{\itLamb}, k_1) + (\II_{\itLamb}, k_2) := (\II_{\itLamb}, k_1+k_2)$, and the $\kk$-linear action
$\kk \times (\{\II_{\itLamb}\} \times \kk) \to \{\II_{\itLamb}\} \times \kk$ by
$(k, (\II_{\itLamb}, k_1) ) \mapsto k(\II_{\itLamb}, k_1):=(\II_{\itLamb}, kk_1).$ Then $\{\II_{\itLamb}\}\times\kk$ is a $\kk$-linear space which is $\kk$-\checkLiu{linearly} isomorphic to $\kk$.
Furthermore, $\{\II_{\itLamb}\}\times\kk$ is also a left $\itLamb$-module where the left $\itLamb$-action
$\itLamb \times (\{\II_{\itLamb}\}\times\kk) \to \{\II_{\itLamb}\}\times\kk$ is defined by $(a, (\II_{\itLamb}, k)) \mapsto a\star(\II_{\itLamb}, k) := \tau(a)(\II_{\itLamb}, k)$ ($=(\II_{\itLamb}, \tau(a)k)$).


\begin{theorem} \label{thm:spC Lambda-homo}
The restriction
\[\spC^{\natural}|_{\im(+)}: \im(+) \to \{\II_{\itLamb}\} \times \kk\]
of $\spC^{\natural}$ is a $\itLamb$-epimorphism.
\end{theorem}

\begin{proof}
First of all, by Proposition \ref{prop:spC mod}, we have known that $\im(+)$ is a left $\itLamb$-submodule of $\im(\spC)$.
Now we show that $\spC^{\natural}|_{\im(+)}$ is \checkLiu{surjective and $\itLamb$-linear}.

For any $a\in \itLamb$ and $(S,r)\in\im(+)$, we have
\begin{align*}
\spC^{\natural}|_{\im(+)}(a\star(S,r))
=\ &  \spC^{\natural}|_{\im(+)}(S,\tau(a)r) = (\II_{\itLamb}, \tau(a)r) \\
=\ & a\star(\II_{\itLamb}, r) = a\star\spC^{\natural}|_{\im(+)}(S,r).
\end{align*}
Since $k=k\tau(1)=\tau(k1)$, it follows that
\begin{center}
$\spC^{\natural}|_{\im(+)}(k(S,r)) = k\spC^{\natural}|_{\im(+)}(S,r)$
\end{center}
where $k\in \kk$ and $1$ is the identity element of $\itLamb$.

On the other hand, for any two elements $(S_1,r_1), (S_2,r_2)\in\im(+)$
(we assume $S_1=[u,v]_{\kk}$ and $S_2=[s,t]_{\kk}$), we have
\begin{align*}
& \spC^{\natural}|_{\im(+)}((S_1,r_1)+(S_2,r_2)) \\
=\ & \spC^{\natural}|_{\im(+)}({S}, r_1+r_2) \\
=\ & (\II_{\itLamb}, r_1+r_2) \\
=\ & (\II_{\itLamb},r_1) + (\II_{\itLamb},r_2) \\
=\ & \spC^{\natural}|_{\im(+)}(S_1,r_1) + \spC^{\natural}|_{\im(+)}(S_2,r_2)
\end{align*}
where ${S} = [\min\{u,v,s,t\}, \max\{u,v,s,t\}]_{\kk}$.
Therefore, $\spC^{\natural}|_{\im(+)}$ is a $\itLamb$-homomorphism.

Since $\dim_{\kk}(\{\II_{\itLamb}\} \times \kk) = 1$, we find that $\spC^{\natural}|_{\im(+)}$ is a $\itLamb$-epimorphism.
\end{proof}

\subsection{Special case: The map $\spC$ in the \checkLiu{case} of $\scrA^1$}
\label{subsect:spC-special}

We assume that $\itLamb=\RR$ and $\scrA^p$ satisfy the L-condition, that is, $\scrA^p=(\scrA^1_{\RR}, \II=[c,d], \kappa_c, \kappa_d, \mu)$ satisfies (L1)--(L6),
then we have $\widetilde{\spI}=\{[t_1,t_2]\mid c\le t_1\le t_2\le d\}$.
In this case, for any $f\in\w{\bfS_{\tau}(\II_{\itLamb})} = \Widehat{\bfS_{\mathrm{id}_{\RR}}([c,d])}$ and $\alpha\in(c,d)$,
let $\spI=\{[\alpha,t]\mid \alpha\le t\le d\}$, we have
\[\w{T}_{\spI}(f) = \{ ([\alpha,t], \w{T}_{\alpha}^{t}(f)) \in \Sigma([c,d])\times\RR \mid \alpha\le t\le d \}.\]

\subsubsection{\sectcolor Restriction $\spC_{\alpha}$ of $\spC$}

\begin{proposition} \label{prop:spC-L}
If $\scrA^1$ satisfies the L-condition, then for any $f\in\Widehat{\bfS_{\mathrm{id}_{\RR}}([c,d])}$ and $\alpha \in (c,d)$, there is a bijection
\begin{align}
  \spC_{\alpha,f}: \frakS_{\alpha,f} =
  \bigg\{ \bigg(t, (\mathrm{L})\int_{\alpha}^{t} f\dd\mu \bigg) \in \RR\times \RR
  \ \bigg|\ \alpha\le t\le d \bigg\}
  \longrightarrow \w{T}_{\spI}(f)
\nonumber
\end{align}
\checkLiu{which sends each pair $(t, (\mathrm{L})\int_{\alpha}^{t} f\dd\mu )$ to the pair $([\alpha, t], \w{T}_{\alpha}^{t}(f))$}.
\end{proposition}

\begin{proof}

Notice that we have
\[ \w{T}_{\spI}(f) = \{ ([\alpha,t], \w{T}_{\alpha}^{t}(f)) \in \Sigma([c,d])\times\RR \mid \alpha\le t\le d \} \subseteq \Sigma([c,d])\times\RR, \]
then for all $([\alpha,t],\w{T}_{\alpha}^t(f)) \in \w{T}_{\spI}(f)$, we have
\[\spC_{\alpha,f}\bigg(t, (\mathrm{L})\int_{\alpha}^{t} f\dd\mu \bigg)
= \spC_{\alpha}|_{\spC_{\alpha,f}} \bigg(t, (\mathrm{L})\int_{\alpha}^{t} f\dd\mu \bigg)
= ([c,d], \w{T}_{\alpha}^{t}(f)) \]
by (\ref{mapsto:spC}).
Thus, $\spC_{\alpha,f}$ is a surjection.

On the other hand, by Lemma \ref{lemm:spC inj}, we know that $\spC$ is injective.
\checkLiu{Since $\spC_{\alpha,f} = \spC|_{\frakS_{\alpha,f}}$ is a retraction of $\spC$}, we have that $\spC_{\alpha,f}$ is a bijection.
\end{proof}

We use the symbol $\biguplus$ to represent a disjoint union. The bijection $\spC_{\alpha,f}$ induces a map
\[  \spC_{\alpha}:
  \frakS_{\alpha}
:= \checkLiu{\biguplus_{f\in\Widehat{\bfS_{\mathrm{id}_{\RR}}([c,d])}}} \frakS_{\alpha,f}
\ -\!\!\!\to\
\w{T}_{\spI}
:= \checkLiu{\biguplus_{f\in\Widehat{\bfS_{\mathrm{id}_{\RR}}([c,d])}}} \w{T}_{\spI}(f)
\]
sending each $(t, (\mathrm{L})\int_{\alpha}^{t} f\dd\mu)$ in $\frakS_{\alpha,f}$ to $([\alpha,t], \w{T}_{\alpha}^{t}(f))$ in $\spC_{\alpha,f}$. Note that the two unions are disjoint unions.
\checkLiu{Hence} $\spC_{\alpha}$ is also a bijection by Proposition \ref{prop:spC-L}.

\subsubsection{\sectcolor Surjection $\spC_{\alpha}^{\natural}\circ\spC_{\alpha}$}
Recall that $\spC^{\natural}\circled\spC$ is a surjection, see Lemma \ref{lemm:spC surj},
and it is straightforward to see that $\spC^{\natural}_{\alpha} \circled \spC_{\alpha}$ is also a surjection,  where $\spC^{\natural}_{\alpha}$ is a restriction of $\spC^{\natural}$ defined as
\[ \spC^{\natural}_{\alpha} := \spC^{\natural}|_{\im(\spC_{\alpha})}:
\im(\spC_{\alpha}) \to \{[c,d]\} \times \RR. \]
We can obtain some properties of $\spC_{\alpha}^{\natural}$ by $\spC_{\alpha}^{\natural}\circ\spC_{\alpha}$.

\begin{lemma} \label{lemm:comm diagram}
The following diagrams
\[\xymatrix{
\frakS_{\alpha} \ar[r]^{\spC_{\alpha}}  \ar[d]_{\spC^{\natural}|_{\frakS_{\alpha}}}
& \w{T}_{\spI} \ar[d]^{\spC^{\natural}|_{\w{T}_{\spI}}} \\
\{[c,d]\}\times\RR \ar[r]_{\mathrm{id}}
& \{[c,d]\}\times\RR
}
\]
and
\[\hspace{-16mm}
\xymatrix@R=2cm@C=2cm{
\frakS_{\alpha}\times \frakS_{\alpha}
\ar[r]^{+ \circled \left(\begin{smallmatrix}
\spC|_{\frakS_{\alpha}} & 0 \\
0 & \spC|_{\frakS_{\alpha}}
\end{smallmatrix}\right)}
\ar[d]_{\left(\begin{smallmatrix}
\spC^{\natural}|_{\frakS_{\alpha}} & 0 \\
0 & \spC^{\natural}|_{\frakS_{\alpha}}
\end{smallmatrix}\right)}
& \frakS_{\alpha} \ar[d]^{\spC^{\natural}|_{\frakS_{\alpha}}} \\
(\{[c,d]\}\times\RR) \times (\{[c,d]\}\times\RR) \ar[r]_{+}
& \{[c,d]\}\times\RR
}
\]
commute.
\end{lemma}

\begin{proof}
For each $(t, (\mathrm{L})\int_{\alpha}^{t} f\dd\mu) \in \frakS_{\alpha}$, we have
\begin{align*}
\spC^{\natural}_{\w{T}_{\spI}}\circled\spC_{\alpha}\bigg(t, (\mathrm{L})\int_{\alpha}^{t} f\dd\mu\bigg)
&= \spC^{\natural}\circled\spC_{\alpha}\bigg(t, (\mathrm{L})\int_{\alpha}^{t} f\dd\mu\bigg) \\
&= \spC^{\natural}([\alpha, t], \w{T}_{\alpha}^{t}(f) ) \\
&= ([c,d], \w{T}_{\alpha}^{t}(f)).
\end{align*}
On the other hand, in the category $\scrA^1$ satisfying L-condition, the following equation
\[ \w{T}_{\alpha}^{t}(f) = \w{T}_{(\RR, t-\alpha, m)}(f) = (\mathrm{L})\int_{\alpha}^{t} f\dd\mu \]
is trivial. Then we have
\begin{align*}
\mathrm{id} \circled \spC^{\natural}|_{\frakS_{\alpha}}\bigg(t, (\mathrm{L})\int_{\alpha}^{t} f\dd\mu\bigg)
&= \spC^{\natural}|_{\frakS_{\alpha}}(t, \w{T}_{\alpha}^{t}(f)) \\
&= ([c,d], \w{T}_{\alpha}^{t}(f)).
\end{align*}
Thus, the first diagram commutes.

For any $(s, (\mathrm{L})\int_{\alpha}^{s} f\dd\mu), (t, (\mathrm{L})\int_{\alpha}^{s} g\dd\mu) \in \frakS_{\alpha}$, $s\ge t$, we have
\begin{align*}
& \spC^{\natural}|_{\frakS_{\alpha}} \circled (+)
\bigg(
{\bigg(\begin{matrix}
\spC|_{\frakS_{\alpha}} & 0 \\
0 & \spC|_{\frakS_{\alpha}}
\end{matrix}\bigg)}
\bigg(
\bigg(s, (\mathrm{L})\int_{\alpha}^{s} f\dd\mu\bigg),
\bigg(t, (\mathrm{L})\int_{\alpha}^{t} g\dd\mu\bigg)
\bigg)
\bigg) \\
=\ & \spC^{\natural}|_{\frakS_{\alpha}}
(([\alpha,s], \w{T}_{\alpha}^{s}(f)) + ([\alpha,t], \w{T}_{\alpha}^{t}(g))) \\
\mathop{=}\limits^{\spadesuit}\ &
\spC^{\natural}|_{\frakS_{\alpha}}
([\alpha,s], \w{T}_{\alpha}^{s}(\id_{[\alpha,s]}f + \id_{[\alpha,t]}g)) \\
=\ & ([c,d], \w{T}_{\alpha}^{s}(\id_{[\alpha,s]}f + \id_{[\alpha,t]}g) )
\end{align*}
and
\begin{align*}
& (+)\circled
{\bigg(\begin{smallmatrix}
\spC^{\natural}|_{\frakS_{\alpha}} & 0 \\
0 & \spC^{\natural}|_{\frakS_{\alpha}}
\end{smallmatrix}\bigg)}
\bigg(
\bigg(s, (\mathrm{L})\int_{\alpha}^{s} f\dd\mu\bigg),
\bigg(t, (\mathrm{L})\int_{\alpha}^{t} g\dd\mu\bigg)
\bigg) \\
=\ & ([c,d], \w{T}_{\alpha}^{s}(f))+([c,d],\w{T}_{\alpha}^{t}(g)) \\
\mathop{=}\limits^{\clubsuit}\ &
([c,d], \w{T}_{\alpha}^{s}(\id_{[\alpha,s]}f + \id_{[\alpha,t]}g) ),
\end{align*}
where $\spadesuit$ is obtained by the definition of addition (\ref{formula:addi-map});
and $\clubsuit$ holds since $\{[c,d]\}\times\RR$ is a $\itLamb$-module (see Theorem \ref{thm:spC Lambda-homo}),
$\w{T}_{\alpha}^{t}(g)=\w{T}_{\alpha}^{s}(\id_{[\alpha,t]}g)$, and $\w{T}_{\alpha}^{s}$ is a $\itLamb$-homomorphism.
Then the second diagram commutes.
\end{proof}


The following proposition provides two properties of $\spC^{\natural}_{\alpha}$, it will help us to provide a categorification of integrals with variable upper limit in Section \ref{sect:Int var upper lim}.

\begin{theorem}  \label{thm:proper}
If $\scrA^p$ satisfies the L-condition, then the following properties hold.
\begin{itemize}
\item[\rm(1)] For any $(\beta, (\mathrm{L})\int_{\alpha}^{\beta} f\dd\mu) \in \frakS_{\alpha, f}$ and the real number $\gamma$ satisfying $\beta\le\gamma\le d$, we have
\[  \spC^{\natural}_{\alpha}\circled\spC_{\alpha}
\bigg(\beta, (\mathrm{L})\int_{\alpha}^{\beta} f\dd\mu \bigg)
+ \spC^{\natural}_{\beta}\circled\spC_{\beta}
\bigg(\gamma, (\mathrm{L})\int_{\beta}^{\gamma} f\dd\mu \bigg)
= \spC^{\natural}_{\alpha}\circled\spC_{\alpha}
\bigg(\gamma, (\mathrm{L})\int_{\alpha}^{\gamma} f\dd\mu \bigg), \]
where the addition in the left of the above formula is defined by {\rm(}\ref{formula:addi-map}{\rm)}
{\rm(}notice that $\{[c,d]\}\times\kk$ needs to be seen as a subset of ${T}_{\spI}(\Widehat{\bfS_{\mathrm{id}_{\RR}}([c,d])})$ in this case{\rm)}.

\item[\rm(2)] $\spC^{\natural}_{\alpha}$
is a $\kk$-linear map, that is, for any $k_1,k_2\in\kk$,
$\tilde{f}:=(\beta, (\mathrm{L})\int_{\alpha}^{\beta} f\dd\mu ) \in \frakS_{\alpha, f}$
and $\tilde{g}:=(\beta, (\mathrm{L})\int_{\alpha}^{\beta} g\dd\mu) \in \frakS_{\alpha, g}$,
we have
\[ \spC^{\natural}_{\alpha}
(k_1 \spC_{\alpha}(\tilde{f})+k_2\spC_{\alpha}(\tilde{g}))
=  k_1\cdot \spC^{\natural}_{\alpha} \circled \spC_{\alpha}(\tilde{f})
+k_2\cdot \spC^{\natural}_{\alpha} \circled \spC_{\alpha}(\tilde{g}). \]
\end{itemize}
\end{theorem}

\begin{proof}
(1) By the definition of $\w{T}_{\spI}$, we have
\[ \spC_{\alpha} \bigg(\beta, (\mathrm{L})\displaystyle\int_{\alpha}^{\beta} f\dd\mu \bigg)
= ([\alpha,\beta], \w{T}_{\alpha}^{\beta}(f)) \]
and
\[ \spC_{\beta} \bigg(\gamma, (\mathrm{L})\displaystyle\int_{\beta}^{\gamma} f\dd\mu \bigg)
= ([\beta,\gamma], \w{T}_{\beta}^{\gamma}(f)). \]
Then, by Lemma \ref{lemm:spcase of addi}, we obtain
\[ ([\alpha,\beta], \w{T}_{\alpha}^{\beta}(f)) + ([\beta,\gamma], \w{T}_{\beta}^{\gamma}(f))
= ([\alpha,\gamma], \w{T}_{\alpha}^{\gamma}(f))
= \spC_{\alpha} \bigg(\gamma, (\mathrm{L})\displaystyle\int_{\alpha}^{\gamma} f\dd\mu \bigg). \]
Thus,
\[  \spC_{\alpha}
\bigg(\beta, (\mathrm{L})\int_{\alpha}^{\beta} f\dd\mu \bigg)
+ \spC_{\beta}
\bigg(\gamma, (\mathrm{L})\int_{\beta}^{\gamma} f\dd\mu \bigg)
= \spC_{\alpha}
\bigg(\gamma, (\mathrm{L})\int_{\alpha}^{\gamma} f\dd\mu \bigg). \]
By the Lemma \ref{lemm:comm diagram}, the statement (1) holds.

(2) One can check that $\im(\spC_{\alpha})$ is a $\kk$-linear subspace of $\im(+)$.
By Theorem \ref{thm:spC Lambda-homo}, $\spC^{\natural}|_{\im(+)}$ is a $\itLamb$-epimorphism,
thus, $\spC^{\natural}|_{\im(+)}$ is a $\kk$-linear map.
Since $\im(\spC_{\alpha})$ is a $\kk$-linear subspace of $\im(+)$,
we obtain that $(\spC^{\natural}|_{\im(+)})|_{\im(\spC_{\alpha})}$ is also a $\kk$-linear map,
so is $\spC_{\alpha}^{\natural} = \spC^{\natural}|_{\im(\spC_{\alpha})} = (\spC^{\natural}|_{\im(+)})|_{\im(\spC_{\alpha})}$.
%
\end{proof}

\section{\sectcolor Integral with variable upper limit} \label{sect:Int var upper lim}
We assume $\kk=\RR$ in this section. The case for $\kk=\mathbb{C}$ is similar.
First, we let $\scrA^{p}$ be the category satisfying the L-condition (assume $\II=[c,d]$ and $\xi=\frac{c+d}{2}$), and $\spI=\{[c,t] \mid c\le t\le d\}$ in this section. We denote by $C_*([c,d])$ the Banach space of continuous functions $F:[0,1]\to \RR, x\mapsto F(x)$ such that $F(c)=0$ with the sup norm. We define
\[\eta: C_*([c,d]) \oplus C_*([c,d]) \to C_*([c,d]) \]
by
\[\eta(F, G)(x)=
\begin{cases}
\displaystyle\frac{1}{2}F(2x-c), & \text{ if } \displaystyle c\le x\le \frac{c+d}{2}; \\
\displaystyle\frac{1}{2}(F(d)+G(2x-d)), & \text{ if } \displaystyle \frac{c+d}{2} \le x \le d.
\end{cases}\]
Then $(C_{*}([c,d]), (\mathrm{id}_{[c,d]}: x\mapsto x), \eta)$ is an object in $\scrA^p$ (due to Mark Meckes, see \cite[Section 2]{Lei2023FA}).
Since $(\Widehat{\bfS_{\mathrm{id}_{\RR}}([c,d])}, d-c, \gamma)$ is an initial object of $\scrA^p$,
the following proposition holds.

\begin{proposition}[Leinster-Meckes, see {\cite[Proposition 2.4]{Lei2023FA}}]
There is a unique $\RR$-linear map
$$\w{T}_{(C_{*}([c,d]), \mathrm{id}_{[c,d]}, \eta)}\in \Hom_{\scrA^1}\big(
(L^1([c,d]), d-c, \gamma),
(C_{*}([c,d]), \mathrm{id}_{[c,d]}, \eta)
\big).$$
It is defined as $f \mapsto \w{T}_{(C_{*}([c,d]), \mathrm{id}_{[c,d]}, \eta)}(f)
= (\mathrm{L})\displaystyle\int_{c}^{x} f\dd\mu$ and it maps each function $f$ in the $L^1$-space $L^1([c,d])$ to its Lebesgue integral with variable upper limit $x$.
\end{proposition}

Second, consider the set $\ImspC_{[c,d]}$ of all sets of the form $\w{T}_{\spI}(f) = \{([c,t], \w{T}_{c}^{t}(f)) \mid c\le t\le d \}$ fixing $c$, $d$ and $\spI$,
i.e., \[\ImspC_{[c,d]} := \big\{\w{T}_{\spI}(f) \mid f\in \Widehat{\bfS_{\mathrm{id}_{\RR}}([c,d])} \big\}
= \w{T}_{\spI}(\Widehat{\bfS_{\mathrm{id}_{\RR}}([c,d])}), \]
one can check that $\ImspC_{[c,d]}$ is an Abelian group, where the addition is defined by
\[  \ImspC_{[c,d]} \times \ImspC_{[c,d]} \to \ImspC_{[c,d]} \]
\begin{center}
$(\w{T}_{\spI}(f), \w{T}_{\spI}(g)) \mapsto \w{T}_{\spI}(f) + \w{T}_{\spI}(g) := \w{T}_{\spI}(f+g), $
\end{center}
and the zero element of $\ImspC_{[c,d]}$ is $\w{T}_{\spI}(0)$. The following lemma shows that $\ImspC_{[c,d]}$ is a $\RR$-linear space.

\begin{lemma} \label{lemm:ImspC linear}
Define the left $\RR$-action
\begin{align*}
\RR\times \ImspC_{[c,d]} & \to \ImspC_{[c,d]} \\
(k,\w{T}_{\spI}(f)) & \mapsto k\w{T}_{\spI}(f) := \w{T}_{\spI}(kf),
\end{align*}
then $\ImspC_{[c,d]}$ is an $\RR$-linear space.
\end{lemma}

\begin{proof}
By Proposition \ref{prop:spC-L}, we find that the map $\spC_{c,f}^{-1}$ from $\w{T}_{\spI}(f)$ to $\frakS_{c,f}$ is a bijection. Furthermore, let $k_{1}=k_{2}=1$ in Theorem \ref{thm:proper} (2), we get that $\spC_{c,f}^{-1}$ is an isomorphism $\w{T}_{\spI}(f) \mathop{-\!\!\!\to}\limits^{\cong} \frakS_{c,f}.$
Then, there is a bijection from $\ImspC_{[c,d]}$ to $S$ induced by $\spC_{c,f}^{-1}$,
\begin{center}
$\Game_{c,f}^{-1} : \ImspC_{[c,d]} \to S := \{\frakS_{c,f} \mid f\in \Widehat{\bfS_{\mathrm{id}_{\RR}}([c,d])}\}$
\end{center}
Thus, we just need to prove that $S$ is an $\RR$-linear space. This proof is divided to the following two parts
\begin{itemize}
\item[(a)] \textbf{Statement 1.}
{\it The set $S$ is an Abelian group, where the addition is defined by
\begin{align*}
S \times S & \to S \\
(\frakS_{c,f},\frakS_{c,g}) & \mapsto \frakS_{c,f} + \frakS_{c,g} := \frakS_{c,f+g}.
\end{align*}}
The addition defined in the above is compatible with the formula in Proposition \ref{thm:proper} (2),
that is, for any $t\in[c,d]$, $(t, (\mathrm{L})\int_{c}^{t} f\dd\mu)\in \frakS_{c,f}$, and $(t, (\mathrm{L})\int_{c}^{t} g\dd\mu) \in \frakS_{c,g}$,
we have
\begin{align*}
& \spC_c^{\natural}\circled\spC_c\bigg(t, (\mathrm{L})\int_{c}^{t} f\dd\mu \bigg)
+ \spC_c^{\natural}\circled\spC_c\bigg(t, (\mathrm{L})\int_{c}^{t} g\dd\mu \bigg) \\
=\ & ([c,d], \w{T}_{c}^{t}f\dd\mu) + ([c,d], \w{T}_{c}^{t}g\dd\mu) \\
\mathop{=}\limits^{\spadesuit}\
& ([c,d], \w{T}_{c}^{t}(f+g)\dd\mu) \\
=\ & \spC_c^{\natural}\circled\spC_c\bigg(t, (\mathrm{L})\int_{c}^{t} (f+g)\dd\mu \bigg),
\end{align*}
where $\spadesuit$ is given by the addition defined on $\im(\spC)$.
Thus, it is easy to see that $\spC_{c,f}^{-1}$ is an isomorphism between two Abelian groups.

\item[(b)] \textbf{Statement 2.}
{\it The set $S$ is an $\RR$-linear space, where the $\RR$-scalar multiplication is defined by
\begin{align*}
\RR \times S & \to S \\
(k,\frakS_{c,f}) & \mapsto k\frakS_{c,f}:= \frakS_{c,kf}.
\end{align*}}
The $\RR$-scalar multiplication is compatible with the formula in Proposition \ref{thm:proper} (2),
that is, for any $t\in[c,d]$, $k\in\RR$, and $(t, (\mathrm{L})\int_{c}^{t} f\dd\mu)\in \frakS_{c,f}$,
we have
\begin{align*}
& \spC_c^{\natural}\bigg( k\spC_c\bigg(t, (\mathrm{L})\int_{c}^{t} f\dd\mu \bigg)\bigg) \\
=\ & \spC_c^{\natural}\big( k ([c,t], \w{T}_{c}^{t}f\dd\mu) \big) \\
\mathop{=}\limits^{\clubsuit}\
& \spC_c^{\natural} ([c,t], \w{T}_{c}^{t}(kf)) \\
=\ & \spC_c^{\natural}\circled\spC\bigg(t, (\mathrm{L})\int_{c}^{t}kf\dd\mu\bigg)
\end{align*}
where $\clubsuit$ holds since $\im(\spC)$ is an $\RR$-linear subspace of $\im(+)$ whose $\RR$-scalar multiplication is defined in Lemma \ref{lemm:im(+) linear} (see Proposition \ref{prop:spC mod}, take $\itLamb=\RR$)
and $\w{T}_{c}^{t}$ is a $\itLamb$-homomorphism (see Theorem \ref{thm:LLHZpre-2}, take $\itLamb=\RR$) in the category $\scrA^1$ with $\II=[c,t]$.
\end{itemize} Therefore, $S$ is a $\RR$-linear space. 
\end{proof}

Finally, we prove the following theorem.

\begin{corollary}[Integral with variable upper limit] \label{thm:main}
We assume that the category $\scrA^p=\scrA^1$ satisfies the L-condition and take $\spI = \{[c, t] \mid c\le t\le d\}$.
As the $\kk$-\checkLiu{linear} bijection given by Proposition \ref{prop:spC-L}, the map $\spC_{\alpha,f}: \ \frakS_{\alpha,f} \to \w{T}_{\spI}(f)$ induces a correspondence
\[H: \ImspC_{[c,d]}
\ -\!\!\!\to C_*([c,d])\]
\begin{center}
$\w{T}_{\spI}(f) \mapsto (\mathrm{L})\displaystyle\int_{c}^{x} f\dd\mu$
\end{center}
such that:
\begin{itemize}
\item[{\rm(1)}] $H$ is an $\RR$-linear map;
\item[{\rm(2)}] there is an $\RR$-linear isomorphism
\[h: \Widehat{\bfS_{\mathrm{id}_{\RR}}([c,d])} \to L^1([c,d]) \]
such that the following diagram
\[\xymatrix{
\Widehat{\bfS_{\mathrm{id}_{\RR}}([c,d])}
\ar[d]_{\w{T}_{\spI}: f\mapsto\w{T}_{\spI}(f)}
\ar[rr]^{h}
&& L^1([c,d]) \ar[d]^{\w{T}_{(C_*([c,d]), d-c, \eta)}}  \\
\ImspC_{[c,d]} \ar[rr]_{H} & & C_*([c,d])
}\]
commutes.
\end{itemize}
\end{corollary}

\begin{proof}
(1) By Lemma \ref{lemm:ImspC linear}, $\ImspC_{[c,d]}$ is an $\RR$-linear space.
It is trivial that $C_{*}([c,d])$ is an $\RR$-linear space.
Now we prove that
\[ \heart := H (k_1\w{T}_{\spI}(f) + k_2\w{T}_{\spI}(g)) = k_1H(\w{T}_{\spI}(f)) + k_2H(\w{T}_{\spI}(g)) =: \diamo \]
holds for all $k_1, k_2\in \RR$ and $\w{T}_{\spI}(f)$, $\w{T}_{\spI}(g)\in \ImspC_{[c,d]}$.
Using Lemma \ref{lemm:ImspC linear} again, we have $k_1\w{T}_{\spI}(f) + k_2\w{T}_{\spI}(g) = \w{T}_{\spI}(k_1f+k_2g)$, then
\begin{align*}
\heart & = H(\w{T}_{\spI}(k_1f+k_2g)) \\
& = (\mathrm{L})\int_{c}^{x} (k_1f+k_2g) \dd\mu \\
& = k_1 \cdot (\mathrm{L})\int_{c}^{x} f \dd\mu + k_2\cdot (\mathrm{L})\int_{c}^{x} g \dd\mu \\
& = k_1 H(\w{T}_{\spI}(f)) + k_2 H(\w{T}_{\spI}(g)) = \diamo
\end{align*}
as required.

(2) Notice that $(\Widehat{\bfS_{\mathrm{id}_{\RR}}([c,d])}, \id_{[c,d]}, \widehat{\gamma}_{\frac{\checks{c+d}}{2}})$
is an initial object of $\scrA^1$ (cf. Theorem \ref{thm:LLHZpre-1}).
On the other hand, $(L^1([c,d]), \id_{[c,d]}, \gamma)$
is also an initial object of $\scrA^1$ (see \cite[Theorem 2.1]{Lei2023FA}), \checks{where} $\gamma$ be the $\kk$-linear map
\[L^1([c,d])\oplus L^1([c,d]) \to L^1([c,d])\]
defined as
\[ (f,g) \mapsto \gamma(f,g) :=
\begin{cases}
f(2x-c), & c \le x < \displaystyle\frac{c+d}{2}; \\
g(2x-d), & \displaystyle\frac{c+d}{2} < x < d.
\end{cases} \] We have the following isomorphism:
\[ h:
(\Widehat{\bfS_{\mathrm{id}_{\RR}}([c,d])}, \id_{[c,d]}, \widehat{\gamma}_{\frac{\checks{c+d}}{2}})
\mathop{\longrightarrow}\limits^{\cong} (L^1([c,d]), \id_{[c,d]}, \gamma)
\]
which induces an $\RR$-isomorphism sending each function $f:[c,d]\to\RR$ to itself.

Therefore, for any $f\in \Widehat{\bfS_{\mathrm{id}_{\RR}}([c,d])}$, we have
\[H\circled\w{T}_{\spI}(f) = H(\w{T}_{\spI}(f)) = (\mathrm{L}) \int_{c}^{x} f\dd\mu
= \w{T}_{(C_{*}([c,d]), d-c, \eta)}(f) \]

\end{proof}

\begin{remark}
(1) The isomorphism $h$ is induced by
\[\Widehat{\bfS_{\mathrm{id}_{\RR}}([c,d])} \mathop{\cong}\limits^{\text{\cite{LLHZpre}}}
\underrightarrow{\lim} E_u \mathop{\cong}\limits^{\text{\cite{Lei2023FA}}} L^1([c,d])\]

(2) Since $h$ is an isomorphism, we have
$\w{T}_{(C_*([c,d]), d-c, \eta)} = H \circled \w{T}_{\spI}\circled h^{-1}$,
then
\[ \im(\w{T}_{(C_*([c,d]), d-c, \eta)})
= \im(H \circled \w{T}_{\spI}\circled h^{-1})
= \im(H \circled \w{T}_{\spI}).  \]
It follows that
\[ H \circled \w{T}_{\spI}\circled h^{-1} :  \Widehat{\bfS_{\mathrm{id}_{\RR}}([c,d])}
\to \im(\w{T}_{(C_*([c,d]), d-c, \eta)}) \]
is surjective, thus the following diagram
\[\xymatrix{
\Widehat{\bfS_{\mathrm{id}_{\RR}}([c,d])}
\ar[d]_{\w{T}_{\spI}: f\mapsto\w{T}_{\spI}(f)}
\ar[rr]^{h}
&& L^1([c,d]) \ar[d]^{\w{T}_{(C_*([c,d]), d-c, \eta)}}  \\
\ImspC_{[c,d]} \ar[rr]_{H} & & \im(\w{T}_{(C_*([c,d]), d-c, \eta)})
}\]
commutes by Corollary \ref{thm:main} (2).
Note that $\w{T}_{\spI}: \Widehat{\bfS_{\mathrm{id}_{\RR}}([c,d])} \to \ImspC_{[c,d]}$ is surjective,
we obtain an $\RR$-linear isomorphism $\ImspC_{[c,d]} \cong \im(\w{T}_{(C_*([c,d]), d-c, \eta)}).$

(3) In particular, $\ImspC_{[c,d]}$ provide a categorification of absolutely continuous function, that is,
all functions lying in $\ImspC_{[c,d]}$ are absolutely continuous in analysis.
\end{remark}

\section{\sectcolor Applications}

\subsection*{\sectcolor Application 1: Basic elementary functions} \label{sect:app1}

\def\rmas{\mathsf{as}}
\def\rms{\mathsf{s}}
\def\rmac{\mathsf{ac}}
\def\rmc{\mathsf{c}}

In this section, we will provide a categorification of basic elementary functions by homomorphisms in the category $\scrA^p$.

\subsection{\sectcolor (Anti-)trigonometric functions and their categorifications} \label{subsect:trigonometric}
\
This subsection is divided into two parts \ref{subsect:trigonometric 1} and \ref{subsect:trigonometric 2}. In the first part, we define two functions $\rmas$ and $\rmac$ by homomorphisms in the category $\scrA^p$,
and show that their inverses $\rms$ and $\rmc$ are periodic functions whose periods are $4K$. In the second part, we show that $2K=\pi$.

\subsubsection{\sectcolor Functions $\rmas$,$\rmac$, $\rms$, and $\rmc$ } \label{subsect:trigonometric 1}
We take \checkLiu{$\kk=\mathbb{C}$ and} $\scrA^p_{\itLamb} = \scrA^1_{\mathbb{C}}  := (\scrA^1_{\mathbb{C}}, [c,d], \kappa_c, \kappa_d, \mu)$ satisfying \checkLiu{the} L-condition.
Let $\spI = \{[0,y] \mid 0\le y \le d\}$. and $f: [c, d] \to \mathbb{C}, t \mapsto \frac{1}{\sqrt{1-t^2}}$.
Now, we consider the following two {\iposet\!\!s}
\[ \w{T}_{\spI}(f)
= \{ ([0,y], \w{T}_{0}^{y}(f))
\ |\  0\le y\le d \}
\in \ImspC_{[0,d]}; \]
and
\[ \w{T}_{\spI}(-\id_{[1,y]}^{\dag}f)
= \{ ([0,y], \w{T}_{0}^{y}(-\id_{[1,y]}^{\dag}f))
\ |\  0\le y\le d \}
\in \ImspC_{[0,d]}, \]
where
\[ \id_{[1,y]}^{\dag} = \begin{cases}
\id_{[1,y]}: t\mapsto  1, & \text{ if } y\ge 1;  \\
-\id_{[y,1]}: t\mapsto -1, & \text{ if } y < 1.
\end{cases} \]
Assume $y\ge 1$. By using the addition defined on the Abelian group $\ImspC_{[0,d]}$, we have
\begin{align*}
\w{T}_{\spI}(f)
+ \w{T}_{\spI}(-\id_{[1,y]}^{\dag}f)
= \w{T}_{\spI}(\id_{[0,1]} f)
= \{([0,y], K) | 1\le y\le d \},
\end{align*}
where $K = \w{T}_0^1(\id_{[0,1]}f) = \w{T}_0^1(f)$ is a constant.
Assume $c\le y\le d$ and define
\[ \w{T}_{y}^{1}(f) =
\begin{cases}
\w{T}_{y}^{1}(f), & \text{ if } y\le 1; \\
-\w{T}_{1}^{y}(f), & \text{ if } y\ge 1.
\end{cases} \]
For any $r,s\in\RR$, we define
\[ [r,s]^{\dag} =
\begin{cases}
[r,s], & \text{ if } r\le s; \\
[s,r], & \text{ if } r\ge s.
\end{cases} \]
Then each pair $([y,1]^{\dag}, \w{T}_{y}^{1}(f))$ in $\w{T}_{\widetilde{\spI}}(f)$ corresponds to a pair $(y, \w{T}_{y}^{1}(f))$ in $\mathbb{C}\times\mathbb{C}$.
Naturally, the pair $(y, \w{T}_{y}^{1}(f))$ induces a function $\rmac: \mathbb{C}\to \mathbb{C}, y\mapsto \w{T}_{y}^{1}(f)=:x$,
and we obtain
\[ x = \w{T}_{y}^{1}(f) = \w{T}_{y}^{-1}(f)+\w{T}_{-1}^{1}(f)
= \w{T}_{y}^{-1}(f) + 2K \]
(It can be proved by (\ref{formula:addi-map}), notice that we need to consider the different values of $y$).
Consider the inverse function $\rmc = \rmac^{-1}$ of $\rmac$, then $y=\rmc(x)$, and we have $\rmc(x-2K) = \rmac^{-1}(\w{T}_{y}^{-1}(f))$.
Since $[y,-1]^{\dag}$ and $[1,-y]^{\dag}$ are symmetric with respect to $0$,
and $f(-t)=f(t)$ holds for all $t\in\mathbb{C}$, it can be checked that
\[ \w{T}_{y}^{-1}(f) - \w{T}_{1}^{-y}(f)
= \w{T}_{y}^{-1}(\frac{1}{\sqrt{1-t^2}}) - \w{T}_{1}^{-y}(\frac{1}{\sqrt{1-t^2}})
= 0 \]
by using Corollary \ref{thm:main}, the $\RR$-linear map $\w{T}$ is a $\itLamb$-homomorphism in category $\scrA^1_{\mathbb{C}}$. Thus,
\[ \rmc(x-2K) = \rmac^{-1}(\w{T}_{y}^{-1}(f)) = \rmac^{-1}(\w{T}_{1}^{-y}(f)) = -y = -\rmc(x). \]
Furthermore, we obtain
\begin{align}\label{formula:rmc}
\rmc(x \pm 2uK) = (-1)^{u}\rmc(x).
\end{align}

We can induce a function $\rmas: \mathbb{C}\to \mathbb{C}, y\mapsto \w{T}_{0}^{y}(f)=:x$ and its inverse function $\rms:=\rmas^{-1}$, and show that
\begin{align}\label{formula:rms}
\rms(x \pm 2uK) = (-1)^{u}\rms(x)
\end{align}
by the similar way.

\subsubsection{\sectcolor The periods of $\rmc$ and $\rms$} \label{subsect:trigonometric 2}
Let $K$ be the element in $\mathbb{C}$ in the formulas (\ref{formula:rmc}) and (\ref{formula:rms}).
Take $\scrA^p_{\itLamb} = \scrA^1_{\RR\times[0,2\pi)} := (\scrA^1_{\RR\times[0,2\pi)}, \II, \kappa_{(R,0)}, \kappa_{(R,2\pi)}, \mu)$, where
\begin{itemize}
\item[(1')] $p=1$;
\item[(2')] $\kk=\itLamb=\RR\times[0,2\pi) = \{(r,\theta)\mid r\in\RR, 0\le\theta<2\pi\}$ ($\cong\mathbb{C}/\langle (0, 2\pi) \rangle$),
the multiplication of $\RR\times[0,2\pi)$ is defined by
\[(r_1,\theta_1)\cdot(r_2,\theta_2):=(r_1r_2, \theta),
\ \theta \equiv (\theta_1+\theta_2) (\bmod 2\pi),
\ 0\le \theta < 2\pi, \]
which is induced by the multiplication of $\mathbb{C}$;
\item[(3')] $\II = \{(R,\theta) \mid 0\le\theta\le 2\pi\}$, $R\in\RR^{+}$, the fully order is defined as
\[ (R,\theta_1) \preceq (R,\theta_2) \text{ if and only if } \theta_1\le \theta_2; \]
\item[(4')] $\xi=(R,\pi)$, $\kappa_{(R,0)}(R,t) = (R,\frac{t}{2})$,
$\kappa_{(R,2\pi)}(R,t) = (R, \frac{t+2\pi}{2})$;
\item[(5')] $\tau: \RR\times [0,2\pi) \to \RR$ is a restriction of $\mathrm{id}_{\RR}$;
\item[(6')] $\mu: \Sigma(\II) \to \RR^{\ge0}$ is the measure defined on $\Sigma(\II)$ sending each arc $c$ to the length of $c$.
\end{itemize}

Then the circumference of the circle $\II$ is equal to $2\pi R$, which can be expressed in terms of the curvilinear integral
\begin{align*}
(\scrA^1_{\RR\times[0,2\pi)})\int_{\II} \dd\mu
= \mathop{\int\mkern-20.8mu \circlearrowleft}\limits_{x^2+y^2=R^2}{\dd s}.
\end{align*}
The right of the above is
\begin{align*}
\mathop{\int\mkern-20.8mu \circlearrowleft}\limits_{x^2+y^2=R^2}{\dd s}
& = \int_{-R}^{R} \sqrt{1+\frac{x^2}{R^2-x^2}} \dd x \\
& = \int_{-R}^{R} \frac{1}{\sqrt{1-(x/R)^2}} \dd x \\
& = R \int_{-1}^{1} \frac{1}{\sqrt{1-t^2}} \dd t = KR.
\end{align*}
Thus, $\displaystyle K=\frac{\pi}{2}$, and the formulas (\ref{formula:rmc}) and (\ref{formula:rms}) are
$\rmc(x+u\pi) = (-1)^u\rmc(x)$ and $\rms(x+u\pi) = (-1)^u\rms(x).$
Furthermore, we have $\rmc(x+2\pi) = \rmc(x) \text{ and }
\rms(x+2\pi) = \rms(x).$
The functions $\rmac$, $\rmas$, $\rmc$ and $\rms$ are called arccosine, arcsine, cosine and sine and are written as ``$\arccos$'', ``$\arcsin$'', ``$\cos$'' and ``$\sin$'' \checkLiu{in classical analysis}, respectively.

\begin{remark}
Since we can define the inverse functions of elliptic integrals with variable upper limit as elliptic functions similarly, the method to define trigonometric functions provides a description of elliptic functions.
\end{remark}

\subsection{\sectcolor Logarithmic and exponential functions} \label{subsect:log and exp}

Let $\scrA^p_{\itLamb} = \scrA^1_{\RR} = (\scrA^1_{\RR}, [1,d], \kappa_1, \kappa_d, \mu)$ satisfy the L-condition.

\subsubsection{\sectcolor Logarithmic functions} \label{subsect:log and exp 1}

Take $c=1$ and consider the map $f: [1,d]\to \RR, 1\mapsto \frac{1}{t}$; the homomorphism $\w{T}_{(\RR, y-1, m)}$ in $\scrA^1_{\RR}$; and $\spI=\{[1,y]\mid 1\le y\le d\}$ in this part.
Then we have a \iposet \[ \w{T}_{\spI}(f) = \{([1,y], \w{T}_1^y(f))\mid 1\le y\le d\} \in \ImspC_{[1,d]}. \]
On the other hand, consider the $\RR$-algebra $\itLamb_0=\RR$ and the category $\scrA^p_{\itLamb_0} = \scrA^1_{\RR} = (\scrA^1_{\RR}, [1,d], \kappa_1, \kappa_d, \mu_0)$
that satisfies the conditions (L1) -- (L5) such that the measure $\mu_0:\Sigma([1,d])\to\RR^{\ge 0}$ sends each interval $[y_1,y_2] \subseteq [1,d]$ to the real number $r(y_2-y_1)$, where $r\in\RR^{>0}$.
The map $\varphi: \RR \to \RR, x\mapsto rx$ is a continuous and non-decreasing monotonic function such that
\[\mu_0 = \varphi\circled\mu: \Sigma([1,d]) \to \RR^{\ge 0} \]
is also a measure, and $\fct$ induces a $\kk$-linear isomorphic
\[\iso: \itLamb_0=\RR \to \itLamb=\RR, \ x\mapsto rx. \]
Then, for all $1\le y\le d$, we have
\[ (\scrA^p_{\itLamb_0})\int_{[\varphi(1),\varphi(y)]} (f\circled\iso)\dd\mu_0
= (\scrA^p_{\itLamb})\int_{[1,y]} f\dd\mu \]
by \cite[Theorem 3.5]{LGWLpre}.
Furthermore, the above equation is equivalent to the following integration by substitution
$$\displaystyle (\text{L-S}) \int_{r}^{ry} (f\circled\iso) \dd\varphi = (\text{L}) \int_{1}^{y} f \dd\mu,$$
cf. \cite[Section 4, the subsections 4.1 and 4.3]{LGWLpre}.
Similarly, we can prove
$$\displaystyle (\text{L-S}) \int_{r}^{ry} (f\circled\iso) \dd\varphi = (\text{L}) \int_{r}^{ry} f \dd\mu$$
using \cite[Corollary 3.6]{LGWLpre}. Thus,
\[\displaystyle(\text{L}) \int_{1}^{y} f \dd\mu = (\text{L}) \int_{r}^{ry} f \dd\mu,\]
i.e.,
\begin{align} \label{log:y trans to ry}
\w{T}_{1}^{y} (f) = \w{T}_{r}^{ry}(f),\ \forall r\in\RR^{>0}, 1\le y \le d.
\end{align}
Furthermore, we have
\begin{align} \label{log:addition}
\w{T}_{1}^{y_1}(f) + \w{T}_{1}^{y_2}(f)
\mathop{=}\limits^{\spadesuit} \w{T}_{1}^{y_1}(f) + \w{T}_{y_1}^{y_1y_2}(f)
\mathop{=}\limits^{\clubsuit} \w{T}_{1}^{y_1y_2}(f),
\end{align}
where, $\spadesuit$ holds by using (\ref{log:y trans to ry}) and $\clubsuit$ holds by using (\ref{formula:addi-map}).
In analysis, the function $x\mapsto \w{T}_{1}^{y}(f)$ induced by the \iposet $\w{T}_{\spI}(f)$ is called a {\defines logarithmic function} and it is denoted by $x=\ln y$. Then (\ref{log:addition}) can be written as
\[ \ln y_1 + \ln y_2 = \ln (y_1y_2). \]
We can also consider the situation for $x\mapsto \w{T}_{y}^{1}(f)$ with $0<y<1$ by the dual method
and obtain the definition of the logarithmic function defined in $\RR^{>0}$.

\subsubsection{\sectcolor Exponential functions} \label{subsect:log and exp 2}
Notice that the function $x\mapsto \ln y:=\w{T}_{1}^{y}(f)$ induces a correspondence $y \mapsto \ln^{-1} x$, which is provided by \iposet $\w{T}_{\spI}(f)$ in the part \ref{subsect:log and exp 1} of Subsection \ref{subsect:log and exp} .
The notation ``$\ln^{-1}$'' represents the inverse of $\ln$ as a mapping $y\mapsto \ln y$.
Then for any $y_1, y_2$ (assume $\ln y_1=x_1$ and $\ln y_2 = x_2$), by (\ref{log:addition}), we have
\begin{align}\label{exp:multi}
\ln^{-1}x_1 \ln^{-1}x_2 = y_1y_2 = \ln^{-1}(\ln y_1 + \ln y_2) = \ln^{-1}(x_1+x_2).
\end{align}

\begin{remark}
In fact, $\ln^{-1}t$ is the function $\mathrm{e}^t$ \checkLiu{in classical analysis}. Thus, (\ref{exp:multi}) is equivalent to
\[ \mathrm{e}^{x_1}\mathrm{e}^{x_2} = \mathrm{e}^{x_1+x_2}, \]
where $\mathrm{e} = \lim\limits_{n\to\infty}(1+\frac{1}{n})^n = 2.718281828\ldots$. As an inverse function of integration with variable upper limit $\displaystyle (\text{L})\int_1^{t} \frac{1}{x} \dd\mu$, $\ln^{-1}t$ provides a definition of $\mathrm{e}^t$ using the category $\scrA^p$, this article does not infer a relationship between $\ln^{-1}t$ and $\mathrm{e}$.
\end{remark}

\subsection*{\sectcolor Application 2: The global dimensions of gentle algebras} \label{sect:app2 gentle}

Recall that a quiver is a quadruple $\Q=(\Q_0,\Q_1,\s,\t)$,
where $\Q_0$ and $\Q_1$ are sets of the {\defines vertex} and {\defines arrow}, the elements are called {\defines vertices} and {\defines arrows}, respectively;
$\s$, $\t$ are two maps $\Q_1\to\Q_0$ sending each element in $\Q_1$ to the element's starting point and ending point, respectively.
A {\defines path of length $l$} on $\Q$ is a sequence of arrows, written as $a_1\cdots a_l$, such that $\t(a_i)=\s(a_{i+1})$ ($1\le i\le l-1$).
\checkLiu{We use $\ell(a_1\cdots a_l)$ to represent the length of $a_1\cdots a_l$.}
Then $\kk\Q$ is the $\kk$-linear space whose basis is the set of all paths on $\Q$, it is a $\kk$-algebra and we call it {\defines path algebra of $\Q$}, its multiplication is defined as
\[ \kk\Q \times \kk\Q \to \kk\Q,\]
\[ (p,q) \mapsto
\begin{cases}
pq, &  \text{ if } \t(p)=\s(q); \\
0,  &  \text{ otherwise. }
\end{cases} \]
In this section, we establish some connections between the integrals and the global dimension of gentle algebras. We obtained two simple and interesting results; see Theorem~\ref{fact:1} and Theorem~\ref{fact:2}.
We always assume that $\kk=\RR$ or $\mathbb{C}$ in this section; and each ideal of $\kk\Q$, written as $\I$, is generated by some linear combinations of paths of length $\ge 2$ such that the quotient $\kk\Q/\I$ of $\kk\Q$ is finite-dimensional.

\subsection{\sectcolor Vertex $\itLamb$-homomorphism $\rmv$ and its restrictions} \label{subsect:rmv}
Let $\itLamb=\kk\Q/\I$ be a finite-dimensional algebra. Then, for any homomorphism $\tau: \itLamb\to\kk$ of $\kk$-algebra, $\itLamb$ has a left $\itLamb$-module structure which is defined by
\[ \itLamb \times \itLamb \to \itLamb, \ (a,x) \mapsto a\star x := \tau(a)x. \]
Notice that $\itLamb$ can be seen as a regular module, that is,
\[ \itLamb \times \itLamb \to \itLamb, \ (a,x) \mapsto ax, \]
Then $\itLamb$ is a $(\itLamb,\itLamb)$-bimodule.
In this subsection, we only consider the case for $\itLamb$ with the left $\itLamb$-module structure.
Thus, $\itLamb$ is a normed left-$\itLamb$-module.

By the definition of path algebra, we have
\[\itLamb = \bigoplus_{l=0}^{+\infty}\bigoplus_{\wp \in\Q_l\backslash\I} \kk \wp + \I, \]
where $\Q_l$ is the set of all paths of length $l$.
The {\defines vertex $\itLamb$-homomorphism} defined on $\itLamb$ is a map $\itLamb \to \kk$ sending each element
$\rmv: \sum\limits_{l=0}^{+\infty} \sum\limits_{\wp\in\Q_1} k_{\wp}\wp + \I$ to the sum $\sum\limits_{v\in\Q_0}k_v$,
where $k_v=k_{e_v}$, and $e_v$ is the path of length zero corresponding to $v\in\Q_0$.
One can check that the vertex $\itLamb$-homomorphism $\rmv$ is a $\itLamb$-homomorphism of $\itLamb$-modules,
where the left $\itLamb$-module structure of $\kk$ is defined as $\itLamb\times\kk \to \kk$, $(a,k)\mapsto a\star k:=\tau(a)k$. For any $a\in\itLamb$, we have the following equation:
\begin{align} \label{formula:in Subsect 6.1}
\rmv\bigg(a \star \sum\limits_{l=0}^{+\infty} \sum\limits_{\wp\in\Q_1}        k_{\wp}\wp + \I \bigg)
& = \rmv\bigg(  \sum\limits_{l=0}^{+\infty} \sum\limits_{\wp\in\Q_1} \tau(a)k_{\wp}\wp + \I \bigg)
= \sum\limits_{v\in\Q_0} \tau(a)k_{\wp} \nonumber \\
& = a \star \sum\limits_{v\in\Q_0} k_{\wp}
= a \star \rmv \bigg(\sum\limits_{l=0}^{+\infty} \sum\limits_{\wp\in\Q_1} k_{\wp}\wp + \I \bigg).
\end{align}

\begin{remark}
Recall that an {\defines augmentation} of an associative algebra $A$ over a field (or a commutative ring) $\kk$ is a homomorphism of $\kk$-algebra $A\to\kk$, see \cite[Chapter VIII]{CE1956} or cf., for example, \cite{LV2012}.
An algebra together with an augmentation is called an {\defines augmented algebra}. For example, a homomorphism defined by $\sum_{g\in G}k_gg \mapsto \sum_{g\in G}k_g$ from a group algebra $A=\kk[G]$ to the field $\kk$ is an augmentation.
One can check that the following properties:
\begin{itemize}
\item[(1)] The composition
\[\xymatrix{ \itLamb=\kk\Q/\I \ar[rr]^{\text{canonical} \atop \text{epimorphism}}
&& \itLamb/\rad(\itLamb) \ar[rr]^{[1\ \cdots\ 1]_{1\times n}}
&& \kk }\]
is a sum of $n$ augmentations.
\item[(2)] Vertex $\itLamb$ homomorphism is the sum of some augmentations.
\end{itemize}
\end{remark}

\subsubsection{\sectcolor $\itLamb$-homomorphisms $\rmv_q$} \label{subsubsect:rmv res}
For any $l\in\NN$ and arbitrary $q = a_1a_2\cdots a_l \in \Q_l$, we define the function $\rmv_q$ as a map
\begin{align*}
\rmv_q:\ \kk\Q/\I = \bigoplus_{l=0}^{+\infty}\bigoplus_{\wp\in\Q_l\backslash\I} \kk\wp  + \I
& -\!\!\!\to \kk, \\
\sum\limits_{l=0}^{+\infty} \sum\limits_{\wp\in\Q_1} k_{\wp}\wp + \I
& \mapsto k_{\s(a_1)} + \cdots + k_{s(a_l)},
\end{align*}
which is induced by the restriction $\rmv\Big|{}_{\bigoplus\limits_{i=1}^{n}\kk e_{\s(a_i)}}$ of $\rmv$.
The following lemma shows that $\rmv_q$ is a homomorphism of the $\kk$-algebras.

\begin{lemma}\label{lemm:vpfunction}
For any path $q$ of the path algebra $\itLamb$, $\rmv_q$ is a $\itLamb$-homomorphism.
\end{lemma}

\begin{proof}
For any two elements
\begin{center}
$\displaystyle x=\sum\limits_{l=0}^{+\infty} \sum\limits_{\wp\in\Q_1} k_{\wp}\wp + \I$ and
$\displaystyle x' =\sum\limits_{l=0}^{+\infty} \sum\limits_{\wp'\in\Q_1} h_{\wp'}\wp' + \I$
($k_{\wp}, h_{\wp} \in \kk$)
\end{center}
and $k, k'\in \kk$, we have
\begin{center}
$\displaystyle \rmv_q(kx) = \sum_{v\in\{e_{\s(a_i)} \mid 1\le i\le l\}} kk_{v} $;
$\displaystyle \rmv_q(k'x') = \sum_{v\in\{e_{\s(a_i)} \mid 1\le i\le l\}} kh_{v} $;

and \ \
$\displaystyle kx+k'x' = \sum_{v\in\Q_0} (kk_{v}+k'h_{v})e_v + \sum_{\ell(\wp)\ge 1} (kk_{\wp}+k'h_{\wp}) \wp$.
\end{center}
Thus,
\[\rmv_q(kx+k'x') = \sum_{v\in\{e_{\s(a_i)} \mid 1\le i\le l\}} (kk_{v}+k'h_{v}) = k\rmv_q(x) + k'\rmv_q(x'). \]
By (\ref{formula:in Subsect 6.1}), we find that $\rmv_q$ is a $\itLamb$-homomorphism.
\end{proof}

\subsubsection{\sectcolor The integrals of $\rmv_q$} \label{subsubsect:rmv res int}
Let $\itLamb$ be a finite-dimensional \checkLiu{$\mathbb{C}$}-algebra whose bound quiver is $\Q$;
Let
$$\scrA^p_{\itLamb} = (\scrA^1_{\itLamb}, [0,1], \kappa_{0}, \kappa_{1}, \mu)$$ such that
$\xi=\frac{1}{2}$; $\kappa_0(x) = \frac{x}{2}$; $\kappa_1(x) = \frac{x+1}{2}$; $\tau:\itLamb\to \checkLiu{\mathbb{C}}$ is a homomorphism of $\checkLiu{\mathbb{C}}$-algebras; and $\mu: \Sigma([0,1]) \to \RR^{\ge 0}$ is a Lebesgue measure;
and assume $\spI = \{[0, t] \mid 0\le t\le 1\}$.
Then for any path $q=a_1\cdots a_l$ in $(\Q, \I)$, we can calculate \checkLiu{the} \iposet
\[ \w{T}_{\spI}(f) = \{ ([0, t], \w{T}_{0}^{t}(f)) \mid 0 \le t\le 1 \} \]
given by the integral of $\rmv_q: \itLamb \to \kk$ as follows ($u_{e_v}$ is the number of all elements in $\{ \s(a_i)=v \mid 1\le i\le l\}$).
\begin{align*}
\w{T}_{\spI}(\rmv_q)
=\ & \bigg([0,t], (\scrA^1_{\itLamb}) \int_{[0,t]_{\itLamb}} \rmv_q \dd\mu \bigg) \\
=\ & \bigg([0,t],
\int_{[0,t]^{\times Z}} \bigg(\sum_{v \in \{\s(a_i) \mid 1\le i\le l\}} u_{e_v}k_{e_v} \bigg) \dd\mu
\bigg)\\
=\ & \bigg([0,t], \sum_{v \in \{\s(a_i) \mid 1\le i\le l\}} u_{e_v}\int_0^t k_{e_v}\dd k_{e_v} \bigg) \\
=\ & \bigg([0,t], \sum_{v \in \{\s(a_i) \mid 1\le i\le l\}} \frac{u_{e_v}t^2}{2} \bigg) \\
=\ & ([0,t], \frac{lt^2}{2}).
\end{align*}
where $Z$ is the number of all elements in $\{\s(a_i) \mid 1\le i\le l\}$ and $\sum\limits_{v \in \{\s(a_i) \mid 1\le i\le l\}}u_{e_v} = l$.
It follows that \checkLiu{the length $\ell(q)$ of $q$ is}
\begin{align}\label{formula:length}
\ell(q) = 2 \cdot (\scrA^1_{\itLamb}) \int_{[0,1]_{\itLamb}} \rmv_q \dd\mu,
\end{align}
where the upper limit of \checkLiu{the} integral is $t=1$.

\subsection{\sectcolor Global dimensions and multiple integrals} \label{subsect:gldim muti int}
In this subsection we provide a description of the global \checkLiu{dimension of a} gentle algebra using multiple integrals.

\subsubsection{\sectcolor Gentle algebras and their Koszul duals} \label{subsubsect:Koszul}
A finite-dimensional algebra $\kk\Q/\I$ is called a gentle algebra if its bound quiver $(\Q,\I)$ is a {\defines gentle pair}, that is, $(\Q,\I)$ satisfies the following conditions.
\begin{itemize}
\item[(1)] For any vertex $v\in\Q_0$, there are at most two arrows ending with $v$ and there are at most two arrows starting with $v$.
\item[(2)] If there are two arrows $\alpha_1,\alpha_2$ and there is an arrow $\beta$ such that $\t(\alpha_1)= \t(\alpha_2) = \s(\beta)$, then at least one of $\alpha_1\beta$ and $\alpha_2\beta$ lies in $\I$, and the other one does not lie in $\I$.
\item[(3)] If there is an arrow $\alpha_1$ and there are two arrows $\beta_1, \beta_2$ such that $\t(\alpha) = \s(\beta_1)=\s(\beta_2)$, then one of $\alpha\beta_1\in\I$ and $\alpha\beta_1\in\I$ lies in $\I$, and the other one does not lie in $\I$.
\item[(4)] Each generator of $\I$ is a path of length two.
\end{itemize}

A finite-dimensional algebra $A=\kk\Q/\I$ is called a {\defines gentle algebra} if its bound quiver $(\Q,\I)$ is a gentle pair.
Gentle algebra has become one of the most popular topics in algebra in recent years; see, for example,
\cite[etc]{ZH2016,CS2023,FGLZ2023,APS2023,CS2023b, QZZ2022,LiuZ2021,ZL2024,LGH2023,Kal2015}.

The {\defines Koszul dual} of a finite-dimensional algebra $A$ is given by the self-extension $\Ext_{A}(\overline{A}, \overline{A})^{\op}$ of $\overline{A}:=A/\rad(A)$ and it is isomorphic to the quadratic dual $A^{!}$.
For a finite-dimensional algebra $A$ of the form $\kk\Q/\I$, there is an explicit description of the quiver and relations of $A^{!}$, see, for instance, \cite{M2007}.
Let $A=\kk\Q/\I$ be a gentle algebra with finite global dimension, then its Koszul dual is isomorphic to $A^{!}=\kk\Q^{!}/\I^{!}$, where
\begin{itemize}
\item $\Q^{!} = (\Q^{!}_0, \Q^{!}_1, \s^{!}, \t^{!})$ such that $\Q^{!}_0=\Q_0$, $\Q^{!}_1=\Q_1$, $\s^{!}: \Q^{!}_1 \to \Q^{!}_0$ sends each element $\alpha \in \Q^{!}_1$ to the vertex $\t(\alpha)\in \Q^{!}_0$,
and $\t^{!}: \Q^{!}_1 \to \Q^{!}_0$ sends each element $\alpha \in \Q^{!}_1$ to the vertex $\s(\alpha)\in \Q^{!}_0$;
\item $\I^{!} = \langle \alpha\beta \mid \beta\alpha \checkLiu{\notin} \I \rangle$.
\end{itemize}
We denoted by $v^!$ the vertex in $\Q^{!}_0$ corresponding to $v\in\Q_0$ and by $\alpha^!$ the arrow in $\Q^{!}_1$ corresponding to $\alpha\in\Q_1$. Then $\I^{!}$ can be written as $\I^{!} = \langle \alpha^{!}\beta^{!} \mid \beta\alpha \checkLiu{\notin} \I\rangle$.


\subsubsection{\sectcolor Global dimensions of gentle algebras and multiple integrals}
\label{subsubset:gldim mult int}
For a gentle algebra $\kk\Q/\I$, we call a path $P=a_1\cdots a_{l}$ in $\Q$ a {\defines permitted thread} if
\begin{itemize}
\item[(p1)] $a_{i}a_{i+1}\notin\I$, $i\le 1\le l-1$;
\item[(p2)] For any arrow $\alpha\in\Q_1$ with $\s(\alpha)=\t(a_{l})$ (resp. $\t(\alpha)=\s(a_1)$),
we have $\alpha a_1\in \I$ (resp. $a_l\alpha\in \I$).
\end{itemize}
Dually, we can define the {\defines forbidden thread} $a_1\cdots a_{l}$ as follows.
\begin{itemize}
\item[(f1)] $a_{i}a_{i+1}\in\I$, $i\le 1\le l-1$;
\item[(f2)] For any arrow $\alpha\in\Q_1$ with $\s(\alpha)=\t(a_{l})$ (resp. $\t(\alpha)=\s(a_1)$),
we have $\alpha a_1\notin \I$ (resp. $a_1\alpha\notin \I$).
\end{itemize}
Let $A=\kk\Q/\I$ be a gentle algebra with finite global dimension.
We denote $\perm(A)$ (resp. $\forb(A)$) by the set of all permitted threads (resp. forbidden thread) on the bound quiver $(\Q,\I)$ of $A$.
From \cite[Theorem 5.10]{LGH2023}, we have the following fact
\begin{align}\label{formula:gent gldim}
\gldim A = \sup\limits_{F\in\forb(A)} \ell(F).
\end{align}
Moreover, one can check the following correspondence
\[\begin{matrix}
\forb(A) & -\!\!\!\to & \perm(A^!)  \\
a_1\cdots a_l & \mapsto & a_l^{!}\cdots a_1^{!}
\end{matrix}\]
is bijective. Then, by (\ref{formula:gent gldim}) and (\ref{formula:length}), we have
\begin{align}
\gldim A
& = \sup\limits_{F\in\forb(A)} \ell(F) \ \
= \sup\limits_{F\in\forb(A)} \ (\scrA^1_{A}) \int_{[0,1]_{A}} 2\rmv_F \dd\mu, \label{mult int 1} 
\\
& = \sup\limits_{P\in\perm(A^{!})} \ell(P)
= \sup\limits_{P\in\perm(A^{!})}(\scrA^1_{A^!}) \int_{[0,1]_{A^!}} 2\rmv_P \dd\mu, \label{mult int 2} 
\end{align}
where $\rmv_P: A^{!} \to \kk$ is the $\itLamb$-homomorphism given in the \ref{subsubsect:rmv res} of Subsection \ref{subsect:rmv};
$\mu$ is a Lebesgue measure defined on the $\sigma$-algebra $\Sigma([0,1]_{A^{!}})$;
and $\scrA^p_{A^!} = (\scrA^p_{A^!}, [0,1], \kappa_0, \kappa_1, \mu)$ is the category given in the \ref{subsubsect:rmv res int} of Subsection \ref{subsect:rmv}.
Two integrals in (\ref{mult int 1}) and (\ref{mult int 2}) can be written as two multiple integrals, that is,
we obtain the following theorem.

Let $\widetilde{l}_F$ (resp. $\widetilde{l}_P$ ) be the number of all elements in the set $\{\s(a_i)\mid 1\le i\le l_F\}$ (resp. the set $\{\s(a_i^{!})\mid 1\le i\le l_P\}$), where $F=a_1\cdots a_{l_F}$ (resp. $P=a_1^{!}\cdots a_{l_P}^{!}$) are the paths in $\forb(A)$ (resp. $\perm(A^!)$).
\begin{theorem} \label{fact:1}
Let $A$ be a gentle algebra with $\gldim A<+\infty$ and $A^!$ be its Koszul dual. Then
\[ \frac{\gldim A}{2}
= \sup_{F\in\forb(A)}   \mathop{\int\cdots\int}\limits_{[0,1]^{\times \widetilde{l}_F}} \rmv_F \dd\mu
= \sup_{P\in\perm(A^!)} \mathop{\int\cdots\int}\limits_{[0,1]^{\times \widetilde{l}_P}} \rmv_P \dd\mu,\]
\end{theorem}

\subsection{\sectcolor  Global dimensions and Stieltjes integrals} \label{subsect:gldim Stie int}
\checkLiu{The Lebesgue-Stieltjes integration generalizes both the Riemann-Stieltjes integrations \cite{Stie1984} and the Lebesgue integrations \cite{L1928} within a broader measure-theoretic framework \cite{Carter2000}. This integral employs the Lebesgue-Stieltjes measure. In this subsection, we provide another description of the global dimensions of gentle algebra using Stieltjes integrals.}

\subsubsection{\sectcolor Stieltjes integrals}
Let $A=\kk\Q/\I$ be a gentle algebra and $A^!=\kk\Q^!/\I^!$ be its Koszul dual, and take two finite-dimensional algebras $\itLamb_1=\kk$ and $\itLamb_2=\kk$.
For any permitted thread $P=\alpha_1\cdots\alpha_l\in\perm(A^!)$, consider injection ${_{P}}\iso:\kk\to\kk$, $x\mapsto \ln x^l$. Then ${_{P}}\iso$ satisfies the following two conditions.
\begin{itemize}
\item[(1)] For any subset $S_1$ of $\itLamb_1$ in $\Sigma(\II_1)$,
the image $S_2=\mathrm{Im}({_{P}}\iso|_{S_1})$ of the restriction ${_{P}}\iso|_{S_1}:S_1\to S_2$ of ${_{P}}\iso$ is also a subset of $\itLamb_2$ in $\Sigma(\II_2)$, where $\II_1=[c_1,d_1]\subseteq\RR$ and $\II_2=[c_2,d_2]\subseteq\RR$ are two intervals; $c_1>0$; $c_2=\ln c_1^l$; and $d_2=\ln d_1^l$.
\item[(2)] The map $\fct:\kk \to \kk$, $x\mapsto \ln x^l$ is a function such that
\begin{align*} 
(\fct\circled\mu_1)(S_1) = (\mu_2\circled\omega|_{S_1})(S_1)\ (=\mu_2(S_2))
\end{align*}
for all $S_1\in\Sigma(\II_1)$, where $\mu_2:\Sigma(\II_2)\to\RR^{\ge 0}$ is a Lebesgue measure.
\end{itemize}

Notice that ${_{P}}\iso$ is an injection; then it has a left inverse which is written as ${_{P}}\isoinv$ in this paper.
One can check that ${_{P}}\isoinv$ preserve measure, and $\fct\circled\mu_1\circled{_{P}}\isoinv|_{\im({_{P}}\iso)}$ is a measure, see \cite[Lemma 3.4]{LGWLpre}.
Moreover, $\varphi: [c_1,d_1]\to\RR$ induced by $\fct$ which send each $x\in[c_1,d_1]$ to $\ln x$ is a monotone non-decreasing and right-continuous function.
Thus, the morphism in $\scrA^1_{\itLamb_1} = (\scrA^1_{\kk}, [c_1, d_1], \kappa_{c_1}, \kappa_{d_1}, \mu_1)$
\[{_{1}\w{T}_{(\kk,1,m)}}:
(\w{\bfS_{\mathrm{id}_{\kk}}(\II_1)}, \id_{\II_1}, \gamma_{\frac{c_1+d_1}{2}})
-\!\!\!\to
(\kk, 1, m) \]
and morphism in $\scrA^1_{\itLamb_2} = (\scrA^2_{\kk}, [c_2, d_2], \kappa_{c_2}, \kappa_{d_2}, \mu_1)$
\[{_{2}\w{T}_{(\kk,1,m)}}:
(\w{\bfS_{\mathrm{id}_{\kk}}(\II_2)}, \id_{\II_2}, \gamma_{\frac{c_2+d_2}{2}})
-\!\!\!\to
(\kk, 1, m) \]
are induced by $\kk$-linear maps which are of the forms $\w{\bfS_{\mathrm{id}_{\kk}}(\II_1)} \to \kk$ and $\w{\bfS_{\mathrm{id}_{\kk}}(\II_2)} \to \kk$ such that
\begin{align}
{{_1\w{T}}_{(\kk,1,m)}} (f)  \nonumber
& = (\scrA^1_{\itLamb_1}) \int_{[c_1,c_2]} f \dd(\fct\circled\mu_1) \nonumber \\
& = (\text{L-S})\int_{c_1}^{d_1} f(x) \dd\varphi
= (\text{L}) \int_{\varphi(c_1)}^{\varphi(d_1)} f({_{P}}\isoinv(x)) \dd\mu_2 \label{formula:LSint} \\
& = (\scrA^2_{\itLamb_2}) \int_{[\varphi(c_1),\varphi(d_1)]} (f\circled{_{P}}\isoinv) \dd\mu_2
= {{_2\w{T}}_{(\kk,1,m)}} (f\circled{_{P}}\isoinv) \nonumber
\end{align}
for all $f\in\w{\bfS_{\mathrm{id}_{\kk}}(\II_1)}$, cf. \cite[Theorem 3.5 or Subection 4.1]{LGWLpre}. Here,

\begin{center}
{\it $\fct\circled\mu_1 \simeq \varphi$ is a Lebesgue-Stieltjes measure. }
\end{center}

For the case for $c_1=1$, and take $\spI_1= \{[1, t] \mid 1\le t\le d_1\}$ and $\spI_2= \{[0, t] \mid 0\le t\le \varphi(d_2)\}$, then (\ref{formula:LSint}) provides the following equation
\begin{align} \label{formula:LSint spcase}
(\text{L-S})\int_{1}^{d_1} f(x) \dd\varphi
= (\text{L}) \int_{0}^{\varphi(d_1)} f({_{P}}\isoinv(x)) \dd\mu_2.
\end{align}

\subsubsection{\sectcolor Global dimensions of gentle algebras and Stieltjes integrals} \label{subsubsect:LSint}
Let $A=\kk\Q/\I$ be a gentle algebra with finite global dimension. Now, consider all permitted threads on the bound quiver $(\Q^{!}, \I^{!})$ of its Koszul dual $A^{!}=\kk\Q^{!}/\I^{!}$.
For any $P=\alpha_1\cdots\alpha_l\in\perm(A^{!})$, we define the map
\begin{align*}
\rmp_P:\ A^{!} = \bigoplus_{l=0}^{+\infty}\bigoplus_{\wp\in\Q_l\backslash\I} \kk\wp^{!}+\I^{!}
&  -\!\!\!\to \kk,  \\
\sum\limits_{l=0}^{+\infty} \sum\limits_{\wp\in\Q_1} k_{\wp^{!}}\wp^{!} + \I^{!}
& \mapsto k_P
\end{align*}
for any $\wp=a_1\cdots a_{l_{\wp}}$ in $\Q_l$ and for any path $\wp^{!}=a_{l_{\wp}}^{!}\cdots a_1^{!}$ in $\Q_l^{!}$.
By the maximality of permitted thread
(see the part \ref{subsubset:gldim mult int}, the condition (p1)),
and by the fact that $\I^{!}$ is generated by some paths of length $\ge 2$, the above is well defined.
One can check that $\rmp_P$ and its restriction $\rmp_P|_{\kk P}: \kk P\cong \kk=\itLamb_1 -\!\!\!\to\kk$ is a $\kk$-linear projection that belongs to $\w{\bfS_{\mathrm{id}_{\kk}}(\II_1)}$. If $f=\rmp_P|_{\kk P}$ in (\ref{formula:LSint}) and (\ref{formula:LSint spcase}), we have
\begin{align*}
{_1\w{T}_{\spI_1}}(\rmp_P|_{\kk P})
=\ \ & \{ ((1,t), {_1\w{T}_{1}^{d_1}}(\rmp_P|_{\kk P})) \mid 1\le t\le d_1 \} \\
\mathop{\longleftrightarrow}\limits^{\spC}_{1-1} \
& \bigg\{ \bigg(t, (\text{L-S})\int_1^t \rmp_P|_{\kk P}(x) \dd\varphi\bigg)
\ \bigg|\ 1\le t\le d_1 \bigg\} \\
=\ \ & \bigg\{ \bigg(\varphi(t), (\text{L})\int_{0}^{\varphi(t)} \rmp_P|_{\kk P}({_{P}}\isoinv(x)) \dd\mu_2\bigg)
\ \bigg|\ 0\le \varphi(t)\le \varphi(d_1) \bigg\} \\
\mathop{\longleftrightarrow}\limits^{\spC}_{1-1} \
& \{ ([0,\varphi(d_1)], {_2\w{T}_{0}^{\varphi(d_1)}}(\rmp_P|_{\kk P}\circled{_{P}}\isoinv)) \mid  0\le \varphi(t)\le \varphi(d_1) \}
= {_2\w{T}_{\spI_2}}(\rmp_P|_{\kk P}\circled{_{P}}\isoinv).
\end{align*}
Thus, for any $x=k_PP\in A^{!}|_{\kk P} \cong \kk$, we have $\rmp_P|_{\kk P}(x)=k_P$, and
\begin{align*}
(\text{L-S})\int_1^{t} k_P \dd\varphi
& = (\text{L})\int_{\varphi(1)}^{\varphi(t)} \rmp_P|_{\kk P}({_{P}}\isoinv(k_P)) \dd\mu_2 \\
& = (\text{R})\int_{0}^{\ln t^l} \mathrm{e}^{\frac{k_P }{l}} \dd k_P = l(t-1),
\end{align*}
where $\displaystyle(\mathrm{R})\int$ represents \checkLiu{Riemann} integration.
It follows that
\[ (\text{L-S})\int_1^2 \rmp_P|_{\kk P} \dd\varphi = l = \ell(\checks{P)} \]
in the case of $t=2$. 

Finally, by (\ref{mult int 2}), we obtain the following theorem.

\begin{theorem} \label{fact:2}
Let $A$ be a gentle algebra with $\gldim A<+\infty$ and $A^!$ be its Koszul dual. Then there is a family of Lebesgue-Stieltjes measures
\begin{align} \label{formula:fact:2}
  \checkLiu{
(\varphi_l: [1,t]\mapsto \ln t^l)_{
   \checks{l\in \{l(P)\mid P\in\perm(A^!) \} }
  }
}
\end{align}
such that
\[ \gldim A =  \sup_{P\in\perm(A^!)} (\mathrm{L\text{-}S})\int_1^2 \rmp_P|_{\kk P} \dd\varphi_\checks{l}. \]
\end{theorem}

\section*{Acknowledgements}

The authors sincerely thank the editors and the referee for their detailed and valuable comments which greatly improve the article.

This work is supported by
the Postgraduate Research and Practice Innovation Program of Jiangsu Province (Grant No. KYCX24\_0123),
the Guizhou Provincial Basic Research Program (Natural Science) (Grant Nos. ZK[2024]YiBan066 and ZD[2025]085),
the Scientific Research Foundation of Guizhou University (Grant Nos. [2023]16, [2022]53 and [2022]65),
and the National Natural Science Foundation of China (Grant Nos. 12171207, 12401042, 62203442).








\def\cprime{$'$}

\end{document}